\documentclass[11pt]{article}
\usepackage{amsmath,amsthm,amsfonts,amssymb,bm,wasysym}
\usepackage{graphicx}
\usepackage{array}
\usepackage{yhmath}
\usepackage{tikz}
\usepackage{hyperref}
\usepackage{fullpage}
\usepackage{verbatim}
\usetikzlibrary{calc}
\usetikzlibrary{arrows}
\usetikzlibrary{patterns}
\usetikzlibrary{decorations.pathreplacing}
\theoremstyle{plain}
\newtheorem{theorem}{Theorem}
\newtheorem{proposition}[theorem]{Proposition}
\newtheorem{corollary}[theorem]{Corollary}
\newtheorem{lemma}[theorem]{Lemma}

\newtheorem{claim}{Claim}

\theoremstyle{definition}

\theoremstyle{remark}
\newtheorem{remark}{Remark}

\begin{document}
\title{A Harris-Kesten theorem for confetti percolation}
\def\A{\mathbb{A}}
\def\Ab{\mathcal{A}b}
\def\absq{{a^{\prime}}^2+{b^\prime}^2}
\def\AP{\text{G}}
\def\app{{a^{\prime\prime}}^2+1}
\def\argmin{\text{argmin}}
\def\arb{arbitrary }
\def\ass{assumption}
\def\arrow{\rightarrow}
\def\codim{\text{codim}}
\def\const{c}
\def\CCG{\text{G}}
\def\colim{\text{colim}}
\def\cond{condition }
\def\C{\mbox{\bf C}}
\def\dell{\partial}
\def\diam{\text{diam}}
\def\E{\mathbb{E}}
\def\Et{\text{Et}}
\def\es{\emptyset}
\def\exp{\text{exp}}
\def\fa{for all }
\def\Fk{\mathcal{F}_{k_0}}
\def\Fm{Furthermore}
\def\G{\mathbb{G}}
\def\gr{\text{gr}}
\def\H{\text{H}}
\def\Hom{\text{Hom}}
\def\Hs{\widetilde{X}_{H,0}}
\def\inj{\hookrightarrow}
\def\id{\text{id}}
\def\iiets{it is easy to see }
\def\iietc{it is easy to check }
\def\Iietc{It is easy to check }
\def\Iiets{It is easy to see }
\def\imp{\Rightarrow}
\def\({\left(}
\def\){\right)}
\def\[{\left[}
\def\]{\right]}
\def\lcu{\left\{}
\def\rcu{\right\}}
\def\im{\mbox{im}}
\def\Inv{\text{Inv}}
\def\Ind{\text{Ind}}
\def\Ip{In particular}
\def\ip{in particular}
\def\LB{\text{LB}}
\def\Lo{\mathcal{L}^o}
\def\mc{\mathcal}
\def\mb{\mathbb}
\def\mf{\mathbf}
\def\Hp{\wt{X}_{H,0}^{'}}
\def\N{\mathbb{N}}
\def\Npo{\mathbf{N}_{\mathcal{P}^o}}
\def\k{\overline{k}}
\def\K{\underline{K}}
\def\ldot{.}
\def\O{\mathcal{O}}
\def\Ob{Observe }
\def\ob{observe }
\def\opartial{\partial^{\text{out}}}
\def\ipartial{\partial^{\text{in}}}
\def\eopartial{\partial^{\text{out}}_{\text{ext}}}
\def\eipartial{\partial^{\text{in}}_{\text{ext}}}
\def\p{\prime}
\def\pp{{\prime\prime}}
\def\Po{\mathcal{P}^0}
\def\P{\mathbb{P}}
\def\Proj{\mbox{\bf P}}
\def\Q{\mathbb{Q}}
\def\QQ{\overline{\Q}}
\def\pr{\text{pr}}
\def\R{\mathbb{R}}
\def\Spec{\text{Spec}}
\def\st{such that }
\def\sl{sufficiently large }
\def\ss{sufficiently small }
\def\sot{so that }
\def\su{suppose }
\def\Su{Suppose }
\def\suf{sufficiently }
\def\udot{\mathaccent\cdot\cup}
\def\Set{\mathcal{S}et}
\def\T{\mathbb{T}}
\def\Twh{Then we have }
\def\te{there exist }
\def\tes{there exists }
\def\tptc{this proves the claim}
\def\Map{\text{Map}}
\def\VLo{\mc{VL}^o}
\def\wt{\widetilde}
\def\Wcon{We conclude }
\def\wcon{we conclude }
\def\wc{we compute }
\def\Wc{We compute }
\def\wo{we obtain }
\def\Z{\mathbb{Z}}
\def\ZSlab{\mathbb{Z}^2_L\times\{0\}^{d-2}}
\author{Christian Hirsch
\thanks{Institute of Stochastics, Ulm University, 89069 Ulm, Germany; E-mail: {\tt christian.hirsch@uni-ulm.de}.}}
\maketitle
\abstract{Percolation properties of the dead leaves model, also known as confetti percolation, are considered. 
More precisely, we prove that the critical probability for confetti percolation with square-shaped 
leaves is $1/2$. This result is related to a question of Benjamini and Schramm concerning disk-shaped leaves and can be seen as a variant of the Harris-Kesten theorem for bond percolation. The proof is based on techniques developed by Bollob\'as and Riordan to determine the critical probability for Voronoi and Johnson-Mehl percolation.}
\section{Introduction}
In recent years much progress has been made to determine the critical value of various two-dimensional configuration models in percolation theory and statistical mechanics that exhibit a more complex dependency structure than classical Bernoulli percolation, see e.g.~\cite{balint09, balint10,bbVoronoi,bbSharp,vdBerg,vdBerg2}. In this paper we consider a spatial percolation process based on the so-called \emph{dead leaves model} which is popular in stochastic geometry, see~\cite{bordDL,deadLeaves}. This model describes the coloring of $\R^2$ observed when covering the plane by black and white leaves according to a space-time Poisson process. A precise definition is given in Section~\ref{defSec}. In percolation literature this process is also known under the name of \emph{confetti percolation} and Benjamini and Schramm have conjectured in~\cite[Problem 5]{confetti} that $p_c=1/2$ for the case of disk-shaped leaves. We will show how the techniques from~\cite{bbVoronoi} and~\cite{bbJM} can be used to prove that the critical probability for square-shaped confetti percolation is precisely $1/2$. 

Let us give a rough outline of the main ideas. As in Bernoulli percolation the part $p_c\geq1/2$ follows from Zhang's elegant proof of $\theta(1/2)=0$. Indeed to apply his method we need to check positive correlation of black-increasing events (this is standard, see e.g.~\cite{bbBook,roy}) and the uniqueness of the infinite cluster. For the latter part one may use the geometric method of~\cite{2dUnique} which has the advantage that no additional discretization is needed.

Proving $p_c\leq1/2$ is less canonical and we follow the framework developed by Bollob\'as and Riordan in~\cite{bbVoronoi,bbJM}. Although one might hope at first that the problem of confetti percolation is considerably simpler than the problem of Voronoi percolation (for instance since the range of dependence is finite), still much work needs to be done to resolve discretization issues. 

In Section~\ref{defSec} we give detailed definitions of planar confetti-type percolation models and introduce further useful notation. The proof of $\theta(1/2)=0$ is provided in Section~\ref{zhangSec}. In fact, this property is shown not only for squares but for a more general class of shapes.
The proof of $p_c\leq1/2$ can be subdivided into two steps that are presented in Sections~\ref{unifBound} and~\ref{unifSec}. In Section~\ref{unifBound} we show that the assumptions of the general RSW-type theorem of Bollob\'as and Riordan (see e.g.~\cite{bbVoronoi}) are satisfied in the confetti model. Together with the sharp-threshold theorem~\cite[Lemma 1]{bbError} (which itself is a variant of the sharp-threshold result~\cite[Theorem 2.1]{friedgutKalai}) this result is used to complete the proof of $p_c\leq1/2$ in Section~\ref{unifSec}. To enhance readability, a crucial coupling construction is postponed to Section~\ref{appendix}.

We believe that a similar approach could be used to consider disk-shaped leaves, although starting from the current article this seems not completely straightforward. \Fm, we would be very interested in considering generalizations using different shapes of sufficiently symmetric leaves or where the leaves are rotated at random. An appealing idea to make the current argument less dependent on the specific shape of the leaf was suggested by an anonymous referee in Remark~\ref{generalApproach}. However, this approach depends on a non-trivial estimate on the tail behavior of the number of visible leaves of the dead leaves model in a bounded window.
\section{Notation and basic definitions}
\label{defSec}
For $(\Omega,\mc{F},\P)$ a probability space and $\{A_s\}_{s\in [0,\infty)}$ a family of events, we say that $A_s$ holds with high probability (short whp) if $\P(A_s)\to1$ as $s\to\infty$. 
For $\rho>0$, $u\in\R^2$ we denote by $Q_\rho(u)=u+\rho[-1/2,1/2]^2$ the square of side length $\rho$ centered at $u$ and write $Q(u)=Q_1(u)$. 

For $\varphi=\{x_n\}_{n\geq1}=\{(z_n,t_n,\sigma_n)\}_{n\geq1}\subset\R^2\times[-1,\infty)\times\{\pm1\}$ locally finite  we will often use the notation $y_n=(z_n,t_n)$ to denote the space-time coordinates of the element $(z_n,t_n,\sigma_n)$. Furthermore, we let $A$ denote a fixed Borel subset of $\R^2$, which we shall later refer to as a `fixed leaf'. We assume that
\begin{itemize}
\item $A$ is invariant with respect to rotations by $\pi/2$ and reflections at the coordinate axes,
\item $A$ is compact, path-connected and contains the origin $o$ in its interior,
\item $A$ is a regular closed set, i.e., it is the closure of its interior, and 
\item $\partial A$ is a Borel subset of $\R^2$ with finite one-dimensional Hausdorff measure.
\end{itemize}
Then the dead leaves process describes a sequence of colored leaves falling onto the plane according to the space-time process $\varphi$. To be more precise, at time $t_n$ a $z_n$-centered leaf appears that is of shape $A$ and color $\sigma_n$ (say black if $\sigma_n=1$ and white if $\sigma_n=-1$). This yields a coloring of the plane by defining the color of a point $u\in\R^2$ to be the color of the first leaf covering $u$ (or undefined if there is either no such leaf or if the color is non-unique). Note that we observe the configuration of leaves from below in order to obtain a static coloring of the entire plane.

To each such locally finite $\varphi$ we can associate a function $\mathsf{height}_\varphi:\R^2\to\R$ mapping a point $u\in\R^2$ to the time the leaf visible at $u$ had arrived. Formally, we put 
$$\mathsf{height}_\varphi(u)=\min(\{t_m:x_m=(z_m,t_m,\sigma_m)\in\varphi, u-z_m\in A\})$$ 
if this set is non-empty and all points $x_m$ assuming this minimum are of the same color, and $\mathsf{height}_\varphi(u)=-2$ otherwise. 
To each point of $\R^2$ we assign a number from $\{0,\pm1\}$ according to the function $\psi_{\varphi}:\R^2\to\{0,\pm1\}$ defined by $\psi(u)=\psi_\varphi(u)=\sigma_n$, where the index $n$ is chosen so that $\mathsf{height}_\varphi(u)=t_n$ and $\psi_\varphi(u)=0$ if $\mathsf{height}_\varphi(u)=-2$. 
Sometimes we also write \emph{$\psi$-black} to describe the attribute of being black in the coloring $\psi$. For instance, the connected components of $\psi^{-1}(1)$ are called \emph{$\psi$-black connected components}. Furthermore, for colorings $\psi_1,\psi_2:\R^2\to\{\pm1\}$ we say 
$\psi_{1}$ \emph{black-dominates} $\psi_{2}$ if $\psi_{1}(x)\geq\psi_{2}(x)$ holds for all $x\in\R^2$. 

One can add a probabilistic flavor to this model by replacing the locally finite set $\varphi$ by an independently $\{\pm1\}$-marked homogeneous Poisson point process $X\subset\R^2\times [0,\infty)$. Furthermore, we write $p=\P(\sigma_n=1)\in(0,1)$ for the probability that a fixed leaf is colored black in $X$. It is easy to see that in the coloring $\psi_X$ with probability $1$ all points of $\R^2$ are colored either black or white. We write $\theta(p,A)$ for the probability that the origin is contained in an unbounded $\psi$-black component. In case that the leaf $A=Q(o)$ is the unit square centered at the origin we also write $\theta(p)$ for $\theta(p,A)$.
 Furthermore, we use the standard definition of \emph{critical probability for percolation}, namely $p_{c,A}=\inf\{p>0:\theta(p,A)>0\}$ and write $p_c$
 in the special case $A=Q(o)$. Note also that it is straightforward to extend the definition of the confetti process so as to allow leaves of randomly varying size and shape.

\section{$\theta(1/2)=0$}
\label{zhangSec}
\subsection{Harris's inequality}
\label{harris}
The basic statement of Harris's inequality is that black-increasing events are positively correlated. Although the classical Harris inequality is stated in a lattice setting, some extra technical work makes it possible to adapt it to the situation of confetti percolation. Similar technical adjustments are explained in detail by Bollob\'as and Riordan in~\cite{bbBook} for the case of Voronoi percolation and we follow their presentation.

An event $E$ defined in terms of two independent Poisson processes $(X^+, X^-)$ is called \emph{black-increasing} if for every configuration $\omega_1=(X_1^+,X_1^-)$ and $\omega_2=(X_2^+,X_2^-)$ with $X_1^+\subset X_2^+$, $X_1^-\supset X_2^-$ and $\omega_1\in E$ we have $\omega_2\in E$. Black-increasing functions are defined similarly.

First, let us consider Harris's inequality when only the process $X^+$ is involved. Let $(\Omega_1,\mc{A}_1,\P_1)$ be the canonical probability space of the random variable $X^+$. Denote by $\Sigma_k$ the $\sigma$-algebra generated by the following information. Set $n=2^k$ and divide $[-n,n]^2\times [0,n]$ in $4n^6$ cubes of side length $1/n$. Decide for each of them whether it contains at least one point of $X^+$ or not. The local finiteness of $X^+$ then implies that $\{\Sigma_k\}_{k\geq1}$ forms a filtration of $(\Omega_1,\mc{A}_1,\P_1)$.

Let $g_1,g_2$ be increasing, bounded and measurable functions. As $\E(g_i|\Sigma_k)$ is an increasing function on the discrete product space determined by $\Sigma_k$, the lattice version of Harris's inequality implies $\E(\E(g_1|\Sigma_k)\E(g_2|\Sigma_k))\geq \E(g_1)\E(g_2)$. As $g_1,g_2$ are bounded, the martingale convergence theorem implies $\E(g_i|\Sigma_k)\xrightarrow{k\to\infty} g_i$ and dominated convergence yields $\E(g_1g_2)\geq \E(g_1)\E(g_2)$. Now we can state a version of  Harris's inequality that can be applied to confetti percolation.
\begin{lemma}
\label{harrisLem}
Let $(\Omega,\mc{A},\P)$ be the canonical product space of $(X^+,X^-)$ and let $B,B^\p\in\mc{A}$ be two black-increasing events. Then $\P(B\cap B^\p)\geq \P(B)\P(B^\p)$.
\end{lemma}
\begin{proof}
Let $f_1=1_B$ and $f_2=1_{B^\p}$. Fixing $X^-$, we obtain $\E(f_1f_2|X^-)\geq\E(f_1|X^-)\E(f_2|X^-)$ and taking expectations yields $\E(f_1f_2)\geq\E(\E(f_1|X^-)\E(f_2|X^-))$. Now observe that $g_i=\E(f_i|X^-)$, $i\in\{1,2\}$ is decreasing in $X^-$. This yields $\E(\E(f_1|X^-)\E(f_2|X^-))\geq\E(f_1)\E(f_2)$. Combining these two inequalities completes the proof of the lemma.
\end{proof}
In particular, if $X\subset \R^2\times[0,\infty)\times\{\pm1\}$ is an independently $\{\pm1\}$-marked homogeneous Poisson point process, then Lemma~\ref{harrisLem} can be applied to $X^+=\{y_n:\sigma_n=1\}$ and $X^-=\{y_n:\sigma_n=-1\}$. Also note that Lemma~\ref{harrisLem} is a special case of~\cite[Theorem 1.4]{fock}.
\subsection{Uniqueness of the unbounded black connected component}
Let $A\subset\R^2$ denote a fixed leaf as described in Section~\ref{defSec}. In this subsection we consider the uniqueness of the unbounded black connected component for confetti percolation. To be more precise, we show the following result.
\begin{proposition}
\label{uniqProp}
Denote by $N$ the (random) number of unbounded black connected components. Then $\P(N=1)=1$ or $\P(N=0)=1$.
\end{proposition}
The proof of Proposition~\ref{uniqProp} is a slight variation to an argument developed by Gandolfi, Keane and Russo in~\cite{2dUnique}. Indeed, their method is based purely on geometric properties of $\R^2$ (such as the Jordan curve theorem) and works just as well in continuous situations. A similar (but in fact more complicated) adaptation was considered in~\cite{harrisPhd}. The proof depends on the following properties of the percolation model:
\begin{enumerate}
\item[(A0)]  The origin lies in the interior of a connected component (i.e., either black connected or white connected) with probability $1$ and it lies in the interior of a white connected component with positive probability.
\item[(A1)] $\P$ is invariant under horizontal and vertical translations, under rotations by $\pi/2$, and under reflections at the coordinate axes.
\item[(A2)] $\P$ is ergodic with respect to (discrete) horizontal translations and, separately, with respect to (discrete) vertical translations.
\item[(A3)] Black-increasing events  are positively correlated.
\item[(A4)] There exists at least one unbounded connected component with positive probability.
\end{enumerate} 
Note that confetti percolation satisfies the first four of these items. For ergodicity, this is a consequence of the mixing property of the three-dimensional homogeneous Poisson point process, see~\cite[Chapter 12.3]{pp1}. In the following, we proceed closely along the original presentation in~\cite{2dUnique}. 
\subsubsection{Preliminaries}
Before we begin with the proof of Proposition~\ref{uniqProp} we collect some preliminary results. \Fm, when dealing with confetti percolation it suffices to consider self-avoiding piecewise linear paths consisting of line segments between points of rational coordinates. We say that a path is \emph{closed} if it starts and ends at the same point. First, we note that it suffices to prove a uniform lower bound for the probability that squares are surrounded by black paths.
\begin{lemma}
\label{boxLem}
Assume \tes $\delta>0$ \st \fa $m\ge1$,
\begin{align}
\label{boxLemEq}
\P\(\text{there exists a closed black path surrounding }Q_m(o)\)\ge\delta.
\end{align}
Then with probability $1$ every bounded subset of $\R^2$ is surrounded by a closed black path. In particular, $\P\(N=1\)=1$.
\end{lemma}
\begin{proof}
If we denote by $A$ the event that for every $m\ge1$ the set $Q_m(o)$ is surrounded by a closed black path, then ergodicity implies $\P\(A\)\in\{0,1\}$. \Fm, as we assumed the existence of $\delta>0$ such that~\eqref{boxLemEq} holds uniformly \fa $m\ge1$, we conclude that $\P\(A\)>0$. Hence, the event $A$ occurs with probability $1$ and every bounded subset of $\R^2$ is surrounded by a closed black path.
\end{proof}
For $m\ge1$ we denote by $H_m=\lcu (z_1,z_2)\in\R^2: \left|z_2\right|\le m \rcu $ the horizontal strip of height $2m$, centered at the $x$-axis. \Fm, for Borel sets $A,B,C\subset\R^2$ with $A\cup B\subset C$ we say that $[A,B;C]$ occurs if \tes a black path in $C$ that starts in $A$ and ends in $B$. Similarly, we say that $[A,\infty;C]$ occurs if $A$ has non-empty intersection with an unbounded connected component of black points in $C$. Then  percolation cannot occur in horizontal strips.

\begin{lemma}
\label{stripLem}
Let $m\ge1$ be arbitrary. Then $\P\([z,\infty; H_m]\)=0$ \fa $z\in H_m$. 
\end{lemma}
\begin{proof}
From Assumption (A0) we conclude that there exists $\varepsilon>0$ \st with positive probability the segment $\{0\}\times [-\varepsilon,\varepsilon]$ is colored white. Stationarity and (A3) therefore imply that with positive probability the entire segment $\{0\}\times[-m,m]$ is white. Hence, denoting by $A$ the event that there exist arbitrarily large $k\ge1$ \st $\{k\}\times[-m,m]$ is white and arbitrarily large $k\ge1$ \st $\{-k\}\times[-m,m]$ is white, the ergodicity assumption implies that $\P\(A\)=1$. The proof is completed by noting that $A$ is contained in the complement of $[z,\infty; H_m]$.
\end{proof}
Denote by $S,T:\R^2\to\R^2$, $S(z)=z+(1,0)$ and $T(z)=z+(0,1)$ the horizontal and vertical translation, respectively. Next, we recall an ergodic result whose proof can be found in~\cite{2dUnique}.
\begin{lemma}[Multiple Ergodic Lemma]
\label{multErgLem}
Let $A_0$, $A_1$ and $A_2$ be monotonic (i.e., either black-increasing or black-decreasing) events. Then
$$D\mbox{-}\text{lim}_{N\to\infty}\P\(A_0\cap S^{-N}A_1\cap S^{-2N}A_2\)=\P\(A_0\)\P\(A_1\)\P\(A_2\),$$
where we write $D\mbox{-}\text{lim}_{N\to\infty}\alpha_N=\alpha$ if there exists a subsequence of density $1$ converging to $\alpha$. 
\end{lemma}

As an important consequence of Lemma~\ref{multErgLem} we obtain lower bounds for percolation probabilities outside bounded Borel sets that are far away from the origin.
\begin{corollary}
\label{multErgCor}
Let $U,V\subset\R^2$ be bounded Borel sets and let $W\subset \R^2$ be an unbounded Borel set. \Fm, let $z\in W$ be arbitrary. Then there exist arbitrarily large integers $N$ \st 
$$\P\(\[z,\infty;W\setminus \(S^{-N}U\cup S^NV\)\]\)\ge \frac 12 \P\([z,\infty;W]\).$$
\end{corollary}
\begin{proof}
Consider the events 
$A_0=\{U\text{ is white}\}$, 
$A_1=[z,\infty;W]$,
$A_2=\{V\text{ is white}\}$,
$\wt{A_N}=[z,\infty;W\setminus \(S^{-N}U\cup S^NV\)]$, 
and note that 
$$S^NA_0\cap A_1\cap S^{-N}A_2=S^NA_0\cap \wt{A_N}\cap S^{-N}A_2.$$
Hence, 
$$\P\(\wt{A_N}\)\P\(S^NA_0\cap S^{-N}A_2\)\ge \P\(S^NA_0\cap A_1\cap S^{-N}A_2\).$$
Then Lemma~\ref{multErgLem} implies 
$$D\mbox{-}\text{lim}_{N\to\infty}\P\(S^{N}A_0\cap A_1\cap S^{-N}A_2\)=\P\(A_0\)\P\(A_1\)\P\(A_2\),$$
and
$$D\mbox{-}\text{lim}_{N\to\infty}\P\(S^{N}A_0\cap  S^{-N}A_2\)=\P\(A_0\)\P\(A_2\).$$
Combining the latter identities yields 
$\P\(\wt{A_N}\)\ge\P\(A_1\)/2,$
for arbitrarily high $N\ge1$.
\end{proof}
The proof of Proposition~\ref{uniqProp} proceeds in two steps. The first step considers the case, where there exists an unbounded black connected component in the upper half-plane of $\R^2$.

\subsubsection{First case: percolation in the upper half-plane}
Denote by $H^+=\{(z_1,z_2)\in\R^2: z_2\ge0\}$ the upper half-plane in $\R^2$. In the first step of the proof we assume that
$\P\([o,\infty;H^+]\)>0$ and denote this probability by $q$. 

Let $m\ge1$ be arbitrary and put $B=Q_m(o)$. Applying Corollary~\ref{multErgCor} with $z=o$, $U=\es$, $V=B$, $W=H^+$ and $T$ instead of $S$ yields 
\begin{align*}
\P\(\[y_{-N},\infty; T^{-N}H^+\setminus B\]\)&=\P\(\[o,\infty; H^+\setminus T^NB\]\) \ge q/2,
\end{align*}
for arbitrarily large integers $N\ge1$. Lemma~\ref{stripLem} implies that 
$\P\(\[y_{-N},\infty;H_N\]\)=0,$
\sot 
$$\P\(\[y_{-N},L_N;H_N\setminus B\]\)\ge q/2,$$
where $L_N=\{(z_1,z_2)\in\R^2:z_2=N\}\subset H^+$ denotes the horizontal line at distance $N$ from the origin. Using the decomposition $L_N=L_N^+\cup L_N^-$ with
$$L_N^\sigma=\lcu (z_1,z_2)\in L_N: \sigma z_1\ge0\rcu,\quad\sigma \in\{+,-\},$$
we obtain 
$$\P\(\[y_{-N},L_N^+;L_N^+\cup H_N\setminus (B\cup L_N)\]\)\ge q/4.$$
Hence, by horizontal reflection, 
$$\P\(\[y_{N},L_{-N}^+;L_{-N}^+\cup H_N\setminus (B\cup L_{-N})\]\)\ge q/4.$$
Defining the event 
$$J=\[y_{-N},L_N^+;L_N^+\cup H_N\setminus (B\cup L_N)\]\cap \[y_{N},L_{-N}^+;L_{-N}^+\cup H_N\setminus (B\cup L_{-N})\],$$
we conclude from (A3) that 
$\P\(J\)\ge q^2/16$.
Next, we note that $J\subset \[y_{-N},y_N;H_N\setminus B\]$. Indeed, denote by $\Gamma$ a path connecting $y_{-N}$ to $L_N^+$. Then $y_N$ and $L_{-N}^+$ are contained in different connected components of $H_N\setminus \Gamma$ and any path $\Gamma^\p$ connecting $y_N$ and $L_{-N}^+$ intersects $\Gamma$. Thus,
$$\P\(\[y_{-N},y_N;H_N\setminus B\]\)\ge q^2/16.$$

Moreover, we note that a path from $y_{-N}$ to $y_N$ in $H_N\setminus B$ partitions $H_N$ into two connected components with the property that \fa sufficiently large $n\ge1$ the sets $[n,\infty)\times[-N,N]$ and $[-n,\infty)\times[-N,N]$ are contained in different connected components. We denote by $J^+$ the event that there exists a black path $\Gamma$ in $H_N\setminus B$ that connects $L_N$ and $L_{-N}$ and \st $B$ and $[-n,\infty)\times[-N,N]$ are contained in different connected components of $H_N\setminus \Gamma$ \fa sufficiently large $n\ge1$. The event $J^-$ is defined similarly. 

From $\P\(J^+\)=\P\(J^-\)\ge q^2/32$
we  obtain
$\P\(J^+\cap J^-\)\ge q^4/1024,$
and, moreover, $J^+\cap J^-$ implies that there exist black paths $\Gamma,\Gamma^\p\subset H_N\setminus B$ \st $B$ is contained in a bounded connected component of $\R^2\setminus \(\Gamma\cup\Gamma^\p\)$. \Ip, an application of Lemma~\ref{boxLem} completes the proof.

\subsubsection{Second case: no percolation in the upper half-plane}
In the second part of the proof we consider the case where percolation does not occur in the upper half-plane, i.e., $\P\([o,\infty;H^+]\)=0.$
We make use of the notion of the \emph{winding number} $i(\Gamma,z)$ of a path $\Gamma\subset\R^2$ around a point $z\in\R^2\setminus \Gamma$. Intuitively it is defined as $1/(2\pi)$ times the total change of angle of the vector $z^\p-z$ as $z^\p$ moves from the starting point of $\Gamma$ to its end point. We refer the reader to~\cite[Chapter 7]{beardon} for a precise definition and elementary properties. 

%
%

In the following we denote by $q$ the probability that the origin is contained in an unbounded black connected component. Due to assumption (A4) we have $q>0$. As before let $m\ge1$ be arbitrary and put $B=Q_m(o)$. Applying Corollary~\ref{multErgCor} with $z=o$, $U=V=B$ and $W=\R^2$, we see that \te arbitrarily large $R\ge1$ \st 
$$\P\(\[x_{2R},\infty;\R^2\setminus \(S^RB\cup S^{3R}B\)\]\)=\P\(\[o,\infty;\R^2\setminus \(S^{-R}B\cup S^{R}B\)\]\)\ge q/2.$$
Similarly, putting $I=[0,4R]\times\{0\}$ and applying Corollary~\ref{multErgCor} with $z=x_{2R}$, $U=V=I$ and $W=\R^2\setminus \(S^RB\cup S^{3R}B\)$ yields 
$$\P\(\[x_{2R},\infty;K_{2R}\]\)\ge q/4,$$
for arbitrarily large $M\ge1$, where we write
$$K_{2R}=\R^2\setminus \(S^RB\cup S^{3R}B\cup S^MI\cup S^{-M}I\).$$
Putting $K=\R^2\setminus \(S^RB\cup S^{-R}B\)$, we also note that the assumption $\P\([o,\infty;H^+]\)=0$ implies that every unbounded black path starting from the origin must intersect the union $S^ML_0^+\cup S^{-M}L_0^-$ infinitely often. Hence, we obtain
\begin{align}
\label{secUniqEq1}
\P\(\[o,S^ML_0^+;K\setminus S^{-M}L_0^-\]\)\ge q/4,
\end{align}
and 
$$\P\(\[x_{2R},S^{-M}L_0^-;K_{2R}\setminus S^{M+4R}L_0^+\]\)\ge q/8.$$
Next, for $\sigma\in\{+,-\}$ denote by $A^\sigma$ the event that there exists a black path $\Gamma$ in $K\setminus S^{-M}L_0^-$ that starts in $o$, ends in $S^ML_0^+$, and satisfies $\sigma\cdot i\(\Gamma, x_{2R}\)>0$. Using
$$A^+\cup A^-=\[o,S^ML_0^+;K\setminus S^{-M}L_0^-\],$$
we conclude from~\eqref{secUniqEq1} that $\mu\(A^+\)=\mu\(A^-\)\ge q/8$. \Ip, 
$$\P\(A^+\cap A^-\cap\[x_{2R},S^{-M}L_0^-;K_{2R}\setminus S^{M+4R}L_0^+\]\)\ge q^3/512.$$
To conclude the proof, it suffices to prove that 
\begin{align}
\label{secUniqEq2}
A^+\cap A^-\cap\[x_{2R},S^{-M}L_0^-;K_{2R}\setminus S^{M+4R}L_0^+\]\subset \[o,x_{2R};\R^2\setminus S^RB\].
\end{align}
Indeed, then similar to the arguments at the end of the previous subsection we can show that the probability that $B$ is surrounded by a closed black path is at least $\(q^3/1024\)^2$. Hence, Lemma~\ref{boxLem} implies the claim.

To prove~\eqref{secUniqEq2} we note that the intersection on the left hand side of~\eqref{secUniqEq2} implies the existence of black paths $\Gamma_+,\Gamma_-\subset \R^2\setminus \(S^RB\cup S^{-M}L_0^-\)$ and $\wt{\Gamma}\subset\R^2\setminus \(S^RB\cup S^ML_0^+\)$ with the following properties:
\begin{enumerate}
\item $\Gamma_+$ and $\Gamma_-$ begin at $o$ and end in $S^ML_0^+$;
\item $i\(\Gamma_+,x_{2R}\)>0$  and $i\(\Gamma_-,x_{2R}\)<0$;
\item $\wt{\Gamma}$ begins at $x_{2R}$ and ends in $S^{-M}L_0^-$. 
\end{enumerate}
\Ip, we can define a (not-necessarily black) closed path $\overline{\Gamma}$ by first using $\Gamma_+$ to get from $o$ to $S^ML_0^+$, then moving along $S^ML_0^+$ to the endpoint of $\Gamma_-$ and finally returning to $o$ by traversing $\Gamma_-$ in the reverse direction. Additivity of the winding number then implies
$$i\(\overline{\Gamma},x_{2R}\)=i\(\Gamma_+,x_{2R}\)-i\(\Gamma_-,x_{2R}\)\ge1,$$
where we also used that the winding number of a closed path is an integer.
Hence, $x_{2R}$ is contained in a bounded connected component of $\R^2\setminus \overline{\Gamma}$, while $S^{-M}L_0^-$ is contained in the unbounded connected component of $\R^2\setminus \overline{\Gamma}$. Since the path $\wt{\Gamma}\subset \R^2\setminus\(S^RB\cup S^{-M}L_0^+\)$ connects $x_{2R}$ and $S^{-M}L_0^-$, it must intersect $\overline{\Gamma}\setminus S^{-M}L_0^+\subset \Gamma_+\cup \Gamma_-$. \Ip, $x_{2R}$ and $o$ can be connected by a black path in $\R^2\setminus S^RB$, \sot the event $\[o,x_{2R};\R^2\setminus S^RB\]$ occurs.

\subsection{Zhang's theorem}
Zhang's famous proof of $\theta(1/2)=0$ for Bernoulli bond percolation on $\Z^2$ is based on rather general arguments and can be generalized to many two-dimensional percolation models with a more complex dependency structure. We provide an explicit proof for the confetti percolation model, closely following the exposition in~\cite{Grim99}.
\begin{proposition}
Let $A\subset\R^2$ denote a fixed leaf as described in Section~\ref{defSec}. Then $\theta(1/2,A)=0$.
\end{proposition}
\begin{proof}
Suppose the contrary, i.e., that $\theta(1/2,A)>0$.
First, for any $n\ge1$ we denote by $A^l(n)$ the event that the left vertical boundary segment of the square $Q_n(o)$ intersects an unbounded black connected component. By $A^b(n)$, $A^r(n)$ and $A^t(n)$ we denote the analogous events corresponding to the bottom, right and top side of $Q_n(o)$, respectively. From ergodicity and our assumption $\theta(1/2,A)>0$, we conclude that there exists an unbounded black connected component with probability $1$. \Ip, 
$$\lim_{n\to\infty}\P\(A^l(n)\cup A^b(n)\cup A^r(n)\cup A^t(n)\)=1.$$

As the events $A^u(n)$, $u\in \{l,b,r,t\}$ are black-decreasing, we conclude from Lemma~\ref{harrisLem} that
\begin{align*}
\P\(\(A^l(n)\cup A^b(n)\cup A^r(n)\cup A^t(n)\)^c\)&\ge\P\(\(A^l(n)\)^c\)^4,
\end{align*}
\sot also $\lim_{n\to\infty}\P\(A^l(n)\)=1$. Therefore, we can choose $n_0\ge1$ with $\P\(A^l(n_0)\)>7/8.$
As $p=1/2$ the latter inequality is equivalent to 
$\P\(A^l_w(n_0)\)>7/8,$
where for $u\in\{l,b,r,t\}$ the event $A^u_w(n_0)$ is defined analogously to the event $A^u(n_0)$, just replacing ``black connected component'' by ``white connected component''. Next, consider the event 
$$A=A^l(n_0)\cap A^r(n_0)\cap A^t_w(n_0)\cap A^b_w(n_0),$$
and note that 
\begin{align}
\label{zhangEq1}
\P\(B^c\)\le \P\(A^l(n_0)^c\)+\P\(A^r(n_0)^c\)+\P\(A^t_w(n_0)^c\)+\P\(A^b_w(n_0)^c\)<1/2.
\end{align}

On the event $B$ there exist at least two unbounded black connected components and two unbounded white connected components in $\R^2\setminus Q_{n_0}(o)$. Both of the unbounded black connected components in $\R^2\setminus Q_{n_0}(o)$ are contained in the unique unbounded connected component $C$ of black points in $\R^2$. Planarity implies that the two unbounded white connected components of $\R^2\setminus Q_{n_0}(o)$ whose existence is guaranteed by $B$ are contained in different connected components of $\R^2\setminus C$. Hence~\eqref{zhangEq1} contradicts Proposition~\ref{uniqProp} that ensures the existence of a unique white connected component with probability $1$.
\end{proof}

\section{An RSW theorem for confetti percolation}
\label{unifBound}
\subsection{Statement of the result}
Again let $A\subset\R^2$ denote a fixed leaf as described in Section~\ref{defSec}. The first step in the proof of $p_c\leq 1/2$ is to check that a general RSW-type theorem as developed by Bollob\'as and Riordan for Voronoi percolation in~\cite{bbVoronoi} (see also \cite[Section 5]{bbSharp}) is true for confetti percolation, too. For $R\subset\R^2$ a rectangle we denote by $H(R)$ the event that there exists a black horizontal crossing of $R$, i.e., a black path in $R$ connecting the left vertical boundary of $R$ to the right vertical boundary of $R$. Note that for confetti percolation it is no restriction of generality to consider only piecewise linear paths. This follows from the regularity assumption on $A$ and the observation that topologically open, connected sets are polygonally connected (see~\cite[Proposition 1.7]{bakNew}). The goal of this section is to prove the following result.
\begin{proposition}
\label{bbRSW}
Let $\rho>1$ be arbitrary and suppose that there exists $c>0$ such that $\P(H([0,s]\times[0,s]))\geq c$ holds for all sufficiently large $s>0$. Then there exists $c^\prime>0$ such that for all $s_0>0$ we can find $s\geq s_0$ with $\P(H([0,\rho s]\times[0,s]))\geq c^\prime$.
\end{proposition}
\begin{remark}
Note that with probability $1$, either $[0,s]\times[0,s]$ contains a black horizontal crossing or a white vertical crossing. \Ip, if $p=1/2$ then invariance under rotation by $\pi/2$ and under the switching of colors shows that $\P\(H([0,s]\times[0,s])\)=1/2$.
\end{remark}
\begin{remark}
Putting $n=\lceil s\rceil$, the relation $[0,\rho n]\subset [0,2\rho s]$ holds for all sufficiently large $s\ge1$, so that Proposition~\ref{bbRSW} also yields the existence of arbitrarily large integers $s\ge1$ with $\P(H([0,\rho s]\times[0,s]))\geq c^\prime$.
\end{remark}
As explained in~\cite[Section 5]{bbSharp} the proof of Proposition~\ref{bbRSW} depends on the following properties of the percolation model:
\begin{enumerate}
\item[(B1)] Black-increasing events are positively correlated (we have seen this in Section~\ref{harris}).
\item[(B2)] The model has the symmetries of $\Z^2$ (true, since our leaves have this property).
\item[(B3)] Disjoint regions are asymptotically independent (in fact for the dead leaves model the range of dependence is finite)
\item[(B4)] For every fixed rectangle $R$ there exists $C>0$ such that the probability that $H(sR)$ holds but every black horizontal crossing of $sR$ has length at least $Cs^{5/2}$ tends to $0$ as $s\to\infty$ (this is not clear \emph{a priori} -- the remainder of this subsection is devoted to this item).
\end{enumerate}
To prove useful asymptotics for the length of a shortest black horizontal crossing we observe that the length of such a crossing is bounded above by the perimeter of the black connected component in which it is contained. Since the boundary of this component is a subset of the union of the boundary of $R$ and the union of the boundaries of (partly) visible leaves intersecting $R$, it suffices to bound the latter.

Choose $r_1,r_2>0$ \st $Q_{r_1}(o)\subset A \subset Q_{r_2}(o)$. The first step is to give a more economical construction of the dead leaves process on a rectangle of the form $R=[0,as]\times[0,bs]$ for $a,b,s>0$. Define $R^{\text{ext}}=[-r_2,as+r_2]\times[-r_2,bs+r_2]$ and denote by $c$ a positive constant to be determined in the following paragraph. \Fm, let $X=\{(y_n,\sigma_n)\}_{1\leq n\leq N}\subset R^{\text{ext}}\times[0,c\log s]\times\{\pm1\}$ be an independently $\{\pm1\}$-marked homogeneous Poisson point process on $R^{\text{ext}}\times[0,c\log s]$ with intensity $1$ whose leaves have constant shape $A$ and satisfy $\P(\sigma_n=1)=p$. Furthermore, denote by $\pi_1: R^{\text{ext}}\times[0,c\log s]\times\{\pm1\}\to R^{\text{ext}}$ the projection onto the first coordinate. Then $\pi_1(X)\subset R^{\text{ext}}$ is a homogeneous Poisson point process on $R^{\text{ext}}$ with intensity $\lambda=c\log s$. Note that to obtain from $X$ a realization of the dead leaves process on the finite rectangle $R$ it suffices to simulate a further $\{\pm1\}$-marked homogeneous Poisson process $X^\prime \subset R^{\text{ext}}\times[c\log s,\infty)\times\{\pm1\}$ and consider the coloring of $R$ induced by the superposition of $X$ and $X^\prime$.

Denote by $B^{}_s$ the event $R\subset\bigcup_{z\in \pi_1(X)}\{z+A\}$ and observe that if $B^{}_s$ occurs then the process $X^\prime$ is not needed to determine the coloring of $R$ induced by $X\cup X^\prime$. We claim that $\P\left(B^{}_s\right)\to1$ as $s\to\infty$. Using the notation $M=\{v\in (r_1/4)\Z^2: \(v+[0,r_1/4]^2\)\cap R\neq \es\}$, we see that $B^{}_s$ is implied by the occurrence of the event $\#(\pi_1(X)\cap (v+[0,r_1/4]^2))\geq1$ for all $v\in M$. However, the probability of its complement is at most ${\left|M\right|}\exp(-\lambda r_1^2/16)\leq {16r_1^{-2}(as+2r_2)(bs+2r_2)}\exp(-\lambda r_1^2/16)$. Choosing a suitable $c$ for which $s^2\exp(-\lambda r_1^2/16)\in O(s^{-1})$ (e.g., for $A=Q(o)$ we choose $c=50$), we see that this expression tends to $0$ as $s\to\infty$. In particular, we conclude that the number of visible squares intersecting $R$ is of order at most $O(s^{5/2})$ whp, thereby verifying the final assumption (B4).

\subsection{Proof of Proposition~\ref{bbRSW}}
For the proof of Proposition~\ref{bbRSW} we closely follow the presentations in~\cite{bbVoronoi} and~\cite{bbBook}, only accomodating the arguments to the case of finite range of dependence. In the present subsection, for $\rho,s>0$ it is convenient to write $f(\rho,s)=\P(H([0,\rho s]\times[0,s]))$. We assume for a contradiction that Proposition~\ref{bbRSW} was false, i.e., that \tes $\rho>1$ with
$\lim_{s\to\infty}f\(\rho,s\)=0.$
\Ip,
\begin{align}
\label{bbRswEq2}
\lim_{s\to\infty}f\(1+\varepsilon,s\)=0,
\end{align}
\fa $\varepsilon>0$. Indeed, as explained in~\cite[Chapter 8.3]{bbBook} positive correlation of black-increasing events yields
\begin{align}
\label{bbRswEq3}
f(a_1+a_2-1,s)&\ge f(a_1,s)f(a_2,s)f(1,s),
\end{align}
for all $a_1,a_2>1$, \sot choosing $k\ge (\rho-1)/\varepsilon$ and applying~\eqref{bbRswEq3} several times ($k-1$ times, to be more precise) implies
$$f(\rho,s)\ge f(1+k\varepsilon,s)\ge f(1+\varepsilon,s)f(1+(k-1)\varepsilon,s)f(1,s)\ge\(f(1+\varepsilon,s)f(1,s)\)^{k-1}f(1+\varepsilon,s).$$
Since $f(\rho,s)$ tends to $0$ as $s\to\infty$, so does $f(1+\varepsilon,s)$.

We note that~\eqref{bbRswEq2} has important implications for the shape of black horizontal crossings in a square. Indeed, for any $\varepsilon>0$, whp all horizontal black crossings of a square with side length $s$ must pass within distance $\varepsilon s$ of the top and bottom sides. Otherwise, we would obtain a black horizontal crossing in an $(s\times (1-\varepsilon)s)$-rectangle. 

In order to illustrate the strength of~\eqref{bbRswEq2}, we consider the following result, see~\cite[Claim 4.3]{bbVoronoi}.
\begin{claim}
\label{bbClaim1}
Let $\varepsilon>0$ be arbitrary and assume that~\eqref{bbRswEq2} holds. Then whp no black path in $[0,s]\times \R$ that starts in $\{0\}\times [-\varepsilon s,\varepsilon s]$ leaves $[0,s]\times [-s/2+2\varepsilon s,s/2+2\varepsilon s]$.
\end{claim}
\begin{proof}
Denoting by $E$ the event that \tes a black path $\Gamma$ in $S^\p=[0,s]\times [-s/2+2\varepsilon s,s/2+2\varepsilon s]$ that starts in $\{0\}\times [-\varepsilon s,\varepsilon s]$ and touches the top side of $S^\p$, it suffices (by symmetry) to show that the complement of $E$ occurs whp.

\Fm, denote by $E_1$ the event that \tes such a path that is contained in the rectangle $R=[0,s]\times[-s/2,s/2+2\varepsilon s]$. Observe that the event $E\setminus E_1$ implies the existence of a black vertical crossing in the rectangle $[0,s]\times [-s/2,s/2+2\varepsilon s]$ and by the observation preceding Claim~\ref{bbClaim1} this is an event whose complement occurs whp. Hence, it suffices to show that $\lim_{s\to\infty}\P\(E_1\)=0$.

Finally, denote by $E_2$ the event that \tes a black path in $R$ starting from an element in $\{0\}\times [\varepsilon s/2,3\varepsilon s/2]$ and ending at $[0,s]\times\{-s/2\}$. By symmetry we note that $\P\(E_1\)=\P\(E_2\)$. As black-increasing events are positively correlated, we see that $\P\(E_1\cap E_2\)\ge \P\(E_1\)^2$. \Fm, if both $E_1$ and $E_2$ occur, then the black paths $\Gamma_1$ and $\Gamma_2$ guaranteed by the events $E_1$ and $E_2$ must intersect and define a black vertical crossing of $[0,s]\times [-s/2,s/2+2\varepsilon s]$. In particular, this yields a contradiction to the discussion preceding the claim.
\end{proof}

Further refinements of such arguments can be used to show that whp every black horizontal crossing of a rectangle contains $16$ disjoint subpaths that define black horizontal crossings of slightly shrunken rectangles. To be more precise, we use the following claim whose proof can be found in~\cite[Claim 4.6]{bbVoronoi}.

\begin{claim}
\label{bbClaim2}
 Fix $C>1/2$ and put $R=R_s=[0,s]\times [-Cs,Cs]$. Moreover, for $j\in\{0,\ldots,4\}$ put $R_j=[js/100,(j+96)s/100]\times [-Cs,Cs]$ and assume that~\eqref{bbRswEq2} holds. Then whp every black horizontal crossing $\Gamma$ of $R$ contains $16$ disjoint subpaths $\Gamma_1,\ldots, \Gamma_{16}$ with the property that for every $i\in\{1,\ldots,16\}$ there exists $j\in\{0,\ldots, 4\}$ \st $\Gamma_i$ crosses $R_j$ horizontally.
\end{claim}

In the final step, we show that the shape of black paths imposed by the claim contradicts our assumption (B4) on the length of shortest black horizontal crossings. Indeed, for $R_s$ an $(s\times 2s)$-rectangle denote by $L(R_s)$ the length of the shortest black horizontal crossing of $R_s$ provided that such a crossing exists (otherwise, we put $L(R_s)=\infty$). Then, for $\eta>0$ a (small) constant, it is convenient to consider the function 
$$ g(s)=\sup\lcu x\ge0: \P\(L(R_s)<x\)\le \eta\rcu.$$
Note that the probability that $L(R_s)<\infty$, i.e., that \tes a black horizontal crossing in $R_s$ is at least $f(1,s)$, \sot by choosing $\eta>0$ sufficiently small we can guarantee that $g(s)\in(0,\infty)$ \fa sufficiently large $s>0$. 

Our next goal is to show $\P\(L(R_s)<16g(0.47s)\)<\eta$, i.e., 
\begin{align}
\label{bbRsw4}
g(s)\ge 16g(0.47s),
\end{align}
for all sufficiently large $s>0$.
Before we show~\eqref{bbRsw4}, let us explain how this inequality can be used in order to derive at the desired contradiction. Iterating~\eqref{bbRsw4} yields $g((1/0.47)^ns)\ge 16^ng(s)$. In particular, there are arbitrarily large $s>0$ \st $g(s)>s^3$. Hence, when $\eta>0$ is chosen such that $\eta <\liminf_{s\to\infty}\P\(L(R_s)<\infty\)/2$, then $\limsup_{s\to\infty}\P\(L(R_s)\in (s^3,\infty)\)>0$, contradicting assumption (B4).

The first step in proving~\eqref{bbRsw4} consists of establishing the inequality
\begin{align}
\label{bbRsw5}
\P\(L(R^\p_s)<g(0.47s)\)\le 2\cdot 10^4\eta^2
\end{align}
\fa sufficiently large $s>0$, where $R^\p_s=[0,0.96s]\times [-s,s]$. Subdivide both the left vertical boundary and the right vertical boundary of $R^\p_s$ into $100$ line segments $\{L_i\}_{1\le i \le 100}$ and $\{L_i^\p\}_{1\le i \le 100}$, respectively, where each of these segments is of length $0.02s$. For fixed $i,j\in\{1,\ldots,100\}$, we say that a black horizontal crossing $\Gamma$ of $R^\p_s$ is \emph{eligible} if $\Gamma$ starts from a point of $L_i$ and ends in a point of $L_j^\p$. As the two families of  line segments cover the vertical boundaries of $R^\p_s$, it suffices to show $\P\(B_{i,j}\)\le 2\eta^2$, where $B_{i,j}$ denotes the event that there exists an eligible path $\Gamma$ with $\nu_1\(\Gamma\)\le g(0.47s)$. 

Consider an eligible path $\Gamma$ starting from $(0,y_0)$ and ending at $(0.96s,y_1)$. Moreover, denote by $\Gamma_0$ the subpath of $\Gamma$ starting from $(0,y_0)$ and ending at the first point where $\Gamma$ touches the vertical line $x=0.47s$. Then, we conclude from Claim~\ref{bbClaim1} that $\Gamma_0$ is contained in $R_0=[0,0.47s]\times[y-0.47s,y+0.47s]$ whp, where $y$ denotes the midpoint of $L_i$. Similarly, denoting by $\Gamma_1$ the subpath of $\Gamma$ starting from the last point where $\Gamma$ touches the vertical line $x=0.49s$ and ending at $(0.96s,y_1)$, we see that $\Gamma_1$ is contained in $R_1=[0.49s,0.96s]\times[y^\p-0.47s,y^\p+0.47s]$ whp, where $y^\p$ denotes the midpoint of $L_j^\p$.  \Ip, we have 
$L(R_0)+L(R_1)\le L(R^\p_s)$
whp. Due to the finite range of dependence, we conclude that for all sufficiently large $s>0$ 
\begin{align*}
\P\(L(R_0)+L(R_1)<g(0.47s)\)&\le \P\(\max\(L(R_0),L(R_1)\)<g(0.47s)\)\\
&=\P\(L(R_0)<g(0.47s)\)\P\(L(R_1)<g(0.47s)\).
\end{align*}
Since the latter expression is at most $\eta^2$, this completes the proof of~\eqref{bbRsw5}.

Having established these preliminaries, we can now complete the proof of~\eqref{bbRsw4}. Using the notation of Claim~\ref{bbClaim2}, we put 
$R=[0,s]\times [-s,s]$ and $R_j=[js/100,(j+96)s/100]\times [-s,s]$, $j\in\{0,\ldots,4\}$. First, note that~\eqref{bbRsw5} implies
$$\P\(\min_{j\in\{0,\ldots,4\}}L(R_j)\ge g(0.47s)\)\ge 1- 10^5\eta^2.$$
Second, Claim~\ref{bbClaim2} yields $L(R)\ge 16\min_{j\in\{0,\ldots,4\}} L(R_j)$ whp, \sot for all sufficiently large $s>0$, 
$$\P\(L(R)\ge 16g(0.47s)\)\ge 1-10^6\eta^2,$$
and the latter expression is at least $1-\eta$, provided that $\eta$ was chosen sufficiently small. In particular, $16g(0.47s)\le g(s)$, as desired.

\section{$p_c\leq1/2$}
\label{unifSec}
In the following we restrict our attention to square-shaped leaves, i.e., $A=Q(o)$. In this section we show how the RSW-type theorem of Section~\ref{unifBound} in conjunction with a suitable sharp-threshold result can be used to obtain $\theta(p)>0$ for $p>1/2$. As in \cite[Section 8.3]{bbBook}, we see that in order to apply this sharp-threshold result, we need to work with a mesh size $\delta(s)$ of the form $\delta(s)\sim s^{-\gamma}$ for some constant $\gamma>0$. Since $\gamma$ may be quite small, we have to expect the occurrence of discretization defects. 

For simplicity of exposition, we assume from now on that $s>0$ is an integer. We denote by $\T^2=\T^2_{10s}$ the two-dimensional torus obtained by the standard glueing of the boundaries of $[0,10s]^2$. Furthermore, it is no problem to construct a confetti process on the torus in the same way as we did for the case of $\R^2$. Even the more economical construction of the process can be transferred. Indeed, write 
$$\lambda=\lambda(s)=50\lfloor \log s \rfloor$$
 and let $X\subset \T^2\times[0,\lambda]\times\{\pm1\}$ be an independently $\{\pm1\}$-marked homogeneous Poisson point process with intensity $1$ whose leaves have constant shape $A=Q(o)$ and satisfy $\P(\sigma_n=1)=p$. 

Note that to obtain from $X$ a realization of the confetti process on $\T^2$ it suffices to independently simulate a further $\{\pm1\}$-marked homogeneous Poisson process $X^\prime\subset\T^2\times[\lambda,\infty)\times\{\pm1\}$ and consider the coloring of $\T^2$ induced by the superposition of $X$ and $X^\prime$. We denote by $C^{(1)}_s$ the event $\T^2\subset\bigcup_{(z,t,\sigma)\in X}{Q_{1/2}(z)}$. In this case the process $X^\prime$ is not needed to determine the coloring of $\T^2$ induced by $X\cup X^\prime$. More precisely, $X^\p$ is still not needed after perturbing the squares in $X$ by at most $1/4$ in each direction. As in Section~\ref{unifBound} one proves that $C^{(1)}_s$ occurs whp.

After these preparations we recall the coupling trick introduced by Bollob\'as and Riordan in \cite[Section 6.3]{bbVoronoi}. Let $p_1,p_2\in(0,1)$ be \st  $p_1<p_2$ and for $i\in\{1,2\}$ denote by 
$X_i=\{(z_{i,n},t_{i,n},\sigma_{i,n})\}_{1\leq n\leq N_i}\subset \T^2\times[0,\lambda]\times\{\pm1\}$ 
two $\{\pm1\}$-marked homogeneous Poisson processes with intensity $1$ and $\P(\sigma_{i,n}=1)=p_i$. Let $\gamma>0$ and write $\delta=\delta(s)=(4\lceil s^{\gamma}/4\rceil)^{-1}$. The value  of $\gamma$ will be specified in the proof of Theorem~\ref{kestenThm}. Furthermore, let $\delta_1=\lceil\delta^{-1/2}\rceil^{-1}$ and write $R_1=[0.25s,8.75s]\times [-0.25s,1.25s]$ and $R_2=[0.5s,8.5s]\times[-0.5s,1.5s]$. For ${\delta_0}>0$ we construct from $\varphi=\{(z_n,t_n,\sigma_n)\}_{1\leq n\leq N}\subset\T^2\times\R\times\{\pm1\} $ a $\{\pm1\}$-marked set $\varphi^{{\delta_0}}$ with the property that black squares are delayed and shrunk, while white ones are advanced and enlarged. More precisely, write $\varphi^{{\delta_0}}=\{(z_n,t_n+\sigma_n{\delta_0},\sigma_n)\}_{1\leq n\leq N}$, where the leaf associated with $(z_n,t_n+\sigma_n\delta_0)$ is given by $Q_{1-2\sigma_n\delta_0}(o)$. 
\begin{lemma}
\label{pubCoupLem}
Fix any $\gamma>0$ and any $p_1,p_2\in(0,1)$ \st $p_1<p_2$. There exists a coupling between $X_1$ and $X_2$ with $\P(E_g)\to1$ as $s\to\infty$, where $E_g$ denotes the intersection of the following events.
\begin{enumerate}
\item $N_1=N_2$ and $\left|y_{1,i}-y_{2,i}\right|_\infty\leq\delta_1$ for all $i\in\{1,\ldots,N_1\}$,
\item $\sigma_{2,i}\geq\sigma_{1,i}$ for all $i\in\{1,\ldots,N_1\}$,
\item there exists a $\psi_{X_2^{2\delta}}$-black horizontal crossing of $R_2$ or there does not exist a $\psi_{X_1}$-black horizontal crossing of $R_1$.
\end{enumerate}
\end{lemma}
The proof of Lemma~\ref{pubCoupLem} is provided in Section~\ref{appendix}. Next, we recall the following variant (that appears in \cite[Lemma 1]{bbError}) of a sharp-threshold result due to Friedgut and Kalai, see \cite[Theorem 2.1]{friedgutKalai}. For $m,n\geq1$ we say that an event $E\subset \{0,\pm1\}^n$ has \emph{symmetry of order} $m$ if it is invariant under the action of a subgroup $\Gamma\subset\Sigma_n$ with the property that all of its orbits consist of at least $m$ elements (here $\Sigma_n$ denotes the symmetric group on $n$ symbols). Furthermore, for $p_b,p_w\in(0,1)$ with $p_b+p_w<1$ we write $\P_{p_b,p_w}$ for the product measure on the space $(\{0,\pm1\}^n,\mc{P}(\{0,\pm1\}^n))$ determined by the marginals $\P_{p_b,p_w}(\omega \in \{1\}\times \{0,\pm1\}^{n-1})=p_b$ and $\P_{p_b,p_w}(\omega \in \{-1\}\times \{0,\pm1\}^{n-1})=p_w$. Here $\mc{P}(\{0,\pm1\}^n)$ denotes the family of all subsets of $\{0,\pm1\}^n$.
\begin{proposition}
\label{FKThm}
There is an absolute constant $c_3>0$ such that if $0<p_b<q_b<1/e$ and $0<q_w<p_w<1/e$, if $E\subset \{0,\pm1\}^n$ is increasing and has symmetry of order $m$, and if $\P_{p_b,p_w}(E)>\eta$, then $\P_{q_b,q_w}(E)>1-\eta$ whenever
$$\min\{q_b-p_b,p_w-q_w\}\geq c_3\log(1/\eta)p_{\max}\log(1/p_{\max})/\log m ,$$
where $p_{\max}=\max\{q_b,p_w\}$.
\end{proposition}
We apply this result in the following situation. 
Let $X=(z_n,t_n,\sigma_n)_{1\leq n\leq N}$ be an independently $\{\pm1\}$-marked homogeneous Poisson point process on $\T^2\times[0,\lambda]$ with intensity $1$ and $\P(\sigma_n=1)=p$. Furthermore, write $Y_b=\{(z_n,t_n):\sigma_n=1\}_{(z_n,t_n,\sigma_n)\in X}$, $Y_w=\{(z_n,t_n):\sigma_n=-1\}_{(z_n,t_n,\sigma_n)\in X}$ and $Y=Y_b\cup Y_w$. Discretize $\T^2\times[0,\lambda]$ into $n_0=n_0(s)=100s^2\lambda\delta^{-3}$ cubes $Q_i$ of the form $Q_i=Q_{\delta}(z)\times [(k-1)\delta,k\delta]$ where $z\in \T^2\cap \delta\Z^2$ and $k\in\{1,\ldots,\lambda\delta^{-1}\}$. We call $Q_i$ \emph{white} if $Q_i\cap Y_w\neq\es$ and \emph{black} if $Q_i\cap Y_b\neq\es$ and $Q_i\cap Y_w=\es$. Otherwise, i.e., if $Q_i\cap Y=\es$, we say that $Q_i$ is \emph{neutral}. Then the states of different cubes are independent and the probabilities that a fixed cube has a specified state are given by
\begin{align*}
p_{\text{white}}&=1-\exp(\delta^3(1-p))\\
p_{\text{black}}&=\exp(-\delta^3(1-p))(1-\exp(-\delta^3p))\\
p_{\text{neutral}}&=\exp(-\delta^3).
\end{align*}
In other words, this probability measure, denoted by $\P^{\T^2}_p$ in the following, is a product measure on $\{0,\pm1\}^{n_0}$ and our goal is to apply Proposition~\ref{FKThm}. The relevant subgroup $\Gamma\subset\Sigma_{n_0}$ is generated by the translations $(x,y,z)\mapsto (x+\delta,y,z)$ and $(x,y,z)\mapsto (x,y+\delta,z)$ of $(\T^2\cap\delta\Z^2)\times((0,\lambda]\cap\delta\Z)$. 

In the following, it will be convenient to associate a coloring $\psi_\omega$ to any discrete configuration $\omega\in\{0,\pm1\}^{n_0}$. Namely define 
$$\varphi_\omega=\{(z,t+(\omega(z,t)-1)\delta/2,\omega(z,t))\}_{(z,t)\in (\delta\Z^2\cap \T^2)\times(\delta\Z\cap(0,\lambda]),\omega(z,t)\in\lcu\pm1\rcu},$$ where the leaf attached to $(z,t+(\omega(z,t)-1)\delta/2,\omega(z,t))$ is given by $Q_{1-2\omega(z,t)\delta}(o))$ and also put $\psi_\omega=\psi_{\varphi_\omega}$.
\begin{lemma}
\label{smallLem}
Let $\varphi,\varphi^\p\in C^{(1)}_s$ be finite subsets that induce the same discrete configuration $\omega\in\{0,\pm1\}^{n_0}$. Then,
\begin{enumerate}
\item $\psi_{\varphi^\p}$ black-dominates $\psi_\omega$, and 
\item $\psi_\omega$ black-dominates $\psi_{\varphi^{2\delta}}$.
\end{enumerate}
\end{lemma}
\begin{proof}
To prove the first claim, let $\zeta\in\R^2$ with $\psi_{\varphi^\p}(\zeta)=-1$ be arbitrary. Next, choose $(z_{n_1}^\p,t_{n_1}^\p,-1)\in\varphi^\p$ with $\zeta\in Q\(z_{n_1}^\p\)$ and \st $t_{n_1}^\p$ is minimal with this property. Also choose $(z,t)\in \delta\Z^2\times\(\delta\Z\cap (0,\lambda]\)$ with $\(z_{n_1}^\p,t_{n_1}^\p\)\in Q_\delta(z)\times [t-\delta,t]$. Now suppose that $\(\wt{z},\wt{t}\)\in\delta\Z^2\times \(\delta\Z\cap (0,\lambda]\)$ is \st $\omega\(\wt{z},\wt{t}\)=1$ and \st $\zeta\in Q_{1-2\delta}(\wt{z})$. Then \tes $\(z_{n_2}^\p,t_{n_2}^\p,1\)\in \varphi^\p$ with $\(z_{n_2}^\p,t_{n_2}^\p\)\in Q_\delta\(\wt{z}\)\times \[\wt{t}-\delta,\wt{t}\]$. From $\zeta\in Q\(z_{n_2}^\p\)$ we conclude $t_{n_2}^\p>t_{n_1}^\p$, \sot $t-\delta\le t_{n_1}^\p<t_{n_2}^\p\le \wt{t}$.

We proceed similarly for the second claim. Let $\zeta\in\R^2$ with $\psi_{\varphi^{2\delta}}(\zeta)=1$ be arbitrary. Next, choose $(z_{n_1},t_{n_1},1)\in\varphi^\p$ with $\zeta\in Q_{1-4\delta}\(z_{n_1}\)$ and \st $t_{n_1}$ is minimal with this property. Also choose $(z,t)\in \delta\Z^2\times\(\delta\Z\cap (0,\lambda]\)$ with $\(z_{n_1},t_{n_1}\)\in Q_\delta(z)\times [t-\delta,t]$. Now suppose that $\(\wt{z},\wt{t}\)\in\delta\Z^2\times \(\delta\Z\cap (0,\lambda]\)$ is \st $\omega\(\wt{z},\wt{t}\)=-1$ and \st $\zeta\in Q_{1+2\delta}(\wt{z})$. Then \tes $\(z_{n_2},t_{n_2},-1\)\in \varphi$ with $\(z_{n_2},t_{n_2}\)\in Q_\delta\(\wt{z}\)\times \[\wt{t}-\delta,\wt{t}\]$. From $\zeta\in Q_{1+4\delta}\(z_{n_2}\)$ we conclude $t_{n_2}-2\delta>t_{n_1}+2\delta$, \sot $t\le t_{n_1}+2\delta<t_{n_2}-2\delta\le \wt{t}-\delta$.
\end{proof}

After these preliminaries the proof of $\theta(p)>0$ for $p>1/2$ is now very similar to \cite[Theorem 1.1]{bbVoronoi} and \cite[Theorem 17]{bbBook} and we only present the main ideas.
\begin{theorem}
\label{kestenThm}
Let $p>1/2$ be arbitrary. Then $\theta(p)>0$.
\end{theorem}
\begin{proof}
Let $p>1/2$ be arbitrary and consider confetti percolation on $\R^2$ with parameter $p$. As in \cite[Theorem 17]{bbBook} it suffices to prove that for all $s_0>0$ there exists $s\geq s_0$ such that the probability of obtaining a black horizontal crossing in a fixed $(3s\times s)$-rectangle is at least $0.99$. In the following, for $q\in(0,1)$ we denote by $\P_q$ the probability measure defined by an independently
$\{\pm1\}$-marked homogeneous Poisson processes on $\T^2\times[0,\lambda]$ with intensity $1$ and whose mark is equal to $1$ with probability $q$.

Let $c>0$ be an absolute (large) constant and write $\gamma=(p-1/2)/c$.
By Proposition~\ref{bbRSW} and the subsequent remark there exists $c_0>0$ with $\P_{1/2}(H([0,9s]\times[0,s]))\geq c_0$ for arbitrarily high values of $s$. We write $R_0=[0,9s]\times[0,s]$ for a fixed $(9s\times s)$-rectangle and $E_1$ for the event that a $\psi_{X}$-black horizontal crossing occurs in $R_0$. Note that the event $E_1$ implies the existence of a $\psi_X$-black horizontal crossing of $R_1$. We denote by $E_2$ the event that there exists a $\psi_{X^{2\delta}}$-black horizontal crossing of $R_2$ and write $p^\prime=(p+1/2)/2$. Then we may apply Lemma~\ref{pubCoupLem} with $p_1=1/2$ and $p_2=p^\p$. As the global event $E_g$ described in that lemma holds whp, we see that the following bound is true for arbitrarily high values of $s$:
\begin{align}
\label{eqLab1}
\P_{p^\prime}(E_2)&\geq\P_{1/2}(E_1)-\P_{1/2}(E^c_g)\geq c_0/2.
\end{align}
Denote by $E_3$ the event in the probability space $(\{0,\pm1\}^{n_0},\mc{P}(\{0,\pm1\}^{n_0}))$ that \emph{some} $(8s\times 2s)$-rectangle in the discrete torus $\delta\Z^2\cap\T^2$ admits a $\psi_\omega$-black horizontal crossing. By part $(ii)$ of Lemma~\ref{smallLem}, \eqref{eqLab1} implies $\P_{p^\prime}^{\T^2}(E_3)\geq c_0/2$.
Note that $E_3$ is a black-increasing event that has symmetries of order $m=(10s/\delta)^2$. Let $\varepsilon=\min\{c_0/2,10^{-100}\}$. Then similarly to~\cite[Theorem 17]{bbBook} for $\delta=(4\lceil s^{\gamma}/4\rceil)^{-1}$ one can show that it is possible to choose $c>0$ \suf large, so that one may apply Proposition~\ref{FKThm} and conclude
$$\P^{\T^2}_p(E_3)\geq1-\varepsilon\geq1-10^{-100}.$$
If we denote by $E_4$ the event that some $(8s\times2s)$-rectangle in $\T^2$ has a black horizontal crossing, then we obtain from part $(i)$ of Lemma~\ref{smallLem} that $E_3\subset E_4$. Now we can follow~\cite[Theorem 17]{bbBook} verbatim to see that for arbitrarily high values of $s$ the occurrence of a black horizontal crossing of a fixed $(12s\times4s)$-rectangle has probability at least $0.99$.
\end{proof}

\section{Proof of Lemma~\ref{pubCoupLem}}
\label{appendix}
In this section we provide a proof of Lemma~\ref{pubCoupLem}. First, we want to formalize the notion of unstable configuration of squares whose connectivity properties can be altered by small perturbations in the space or time dimension.
Since the notation $\delta=\delta(s)\sim s^{-\gamma}$ is reserved for the mesh size corresponding to the discretization we are about to construct, we use the notation $\delta_0$ for a temporary positive variable that could take the value of $\delta(s)$ or some related quantity. Furthermore, denote by $\pi_1: \T^2\times[0,\lambda]\times\{\pm1\}\to \T^2$ and $\pi_{1,2}: \T^2\times[0,\lambda]\times\{\pm1\}\to \T^2\times[0,\lambda]$ the projection onto the first coordinate and the first two coordinates, respectively. Then $\pi_1(X)\subset \T^2$ is a homogeneous Poisson point process on $\T^2$ with intensity $\lambda$ and $\pi_{1,2}(X)$ is a homogeneous Poisson point process on $\T^2\times[0,\lambda]$ with intensity $1$. For ${\delta_0}>0$ and $z,z^\prime\in \pi_1(X)$, $z\neq z^\prime$ we say that $\{z,z^\p\}$ forms a \emph{spatially} ${\delta_0}$\emph{-unstable pair} if and only if $z-z^\p\in A_{{\delta_0}}$, where we write 
$$A_{{\delta_0}}=\{z\in\T^2: z\in Q_{2\delta_0}(o)\oplus((\{0,\pm 1\}\times[-2,2])\cup ([-2,2]\times \{0,\pm 1\}))\},$$
and where for any $A,B\subset \R^2$ we put $A\oplus B=\lcu a+b:\, a\in A, b\in B\rcu$.
See Figure~\ref{fig1} for an illustration of the set $A_{\delta_0}$.
\begin{figure}[!htpb]
\centering
\begin{tikzpicture}[scale=0.25]
\fill[pattern=north west lines, pattern color=black]
(-16,-0.5) rectangle (16,0.5);
\fill[pattern=north west lines, pattern color=black]
(-0.5,-16) rectangle (0.5,16);

\fill[pattern=north west lines, pattern color=black]
(7,-16) rectangle (8,16);
\fill[pattern=north west lines, pattern color=black]
(-8,-16) rectangle (-7,16);

\fill[pattern=north west lines, pattern color=black]
(-16,-8) rectangle (16,-7);
\fill[pattern=north west lines, pattern color=black]
(-16,7) rectangle (16,8);

\draw (-4,-4) rectangle (4,4);

\coordinate[label=below:$o$] (x1) at (0,0);
\coordinate[label=below:$2{\delta_0}$] (x2) at (-6.5,-8);
\coordinate[label=below:$1$] (x3) at (3.8,-8);

\draw[decorate,decoration={brace,mirror}] (-8,-8)--(-7,-8);
\draw[decorate,decoration={brace,mirror}] (0,-8)--(7.5,-8);
\fill (0,0) circle (5pt);
\end{tikzpicture}
\caption{The square $Q(o)$ (black) and the set $A_{{\delta_0}}$ (hatched)}
\label{fig1}
\end{figure}
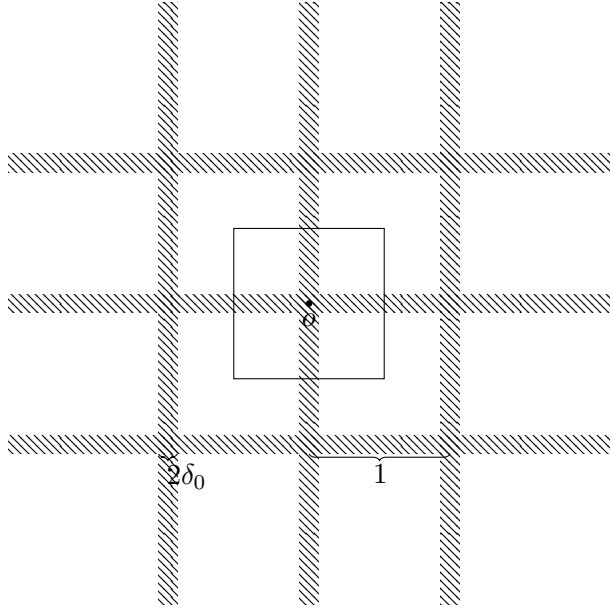

Similarly, we can define instabilities with respect to time. For $t,t^\p>0$, $t\neq t^\p$, we say that $\{t,t^\p\}$ forms a \emph{temporally} ${\delta_0}$\emph{-unstable pair} if $\left|t-t^\p\right|<{\delta_0}$. 

Finally, for $y=(z,t),\,y^\p=(z^\p,t^\p)\in\T^2\times[0,\lambda]$ we say that $\{y,y^\p\}$ forms a ${\delta_0}$\emph{-unstable pair} if the following two properties hold:
\begin{enumerate}
\item $\left|y-y^\p\right|_\infty\leq 2+{\delta_0}$ and
\item $\{z,z^\p\}$ forms a spatially ${\delta_0}$-unstable pair or $\{t,t^\p\}$ forms a temporally ${\delta_0}$-unstable pair.
\end{enumerate}
 As already mentioned above, due to the coarseness of the discretization we \emph{cannot} exclude the existence of $\delta$-unstable pairs whp. We first note that the total number of $\delta$-unstable pairs in a $\log s$-square is bounded above by a constant whp.
\begin{lemma}
\label{etaLem1}
Let $X$ be as above, let ${\gamma_0}>0$ be arbitrary and let ${\delta_0}={\delta_0}(s)\leq s^{-{\gamma_0}}$. Then \tes $K>0$ \st the probability that the total number of $\delta_0$-unstable pairs contained in $Q_{\log s}(o)\times[0,\lambda]$ is larger than $K$ is of order $O\(s^{-3}\)$.
\end{lemma}
\begin{proof}
 By the Poisson concentration inequality (see e.g. \cite[Lemma 1.2]{penrose}) there exist constants $c_1,K_1>0$ such that for $M=K_1\(\log s\)^3$ we have $\P\(\#(\pi_1(X)\cap(Q_{\log s}(o)))\geq M\)\leq c_1s^{-3}$. Let $y_1,\ldots,y_M$ be iid points placed uniformly at random in $Q_{\log s}(o)\times[0,\lambda]$.
Next, denote by 
$$A^\p=\bigcup_{j=1}^M\lcu \left|\lcu k\in\{1,\ldots,M\}:\,\{y_j,y_k\}\text{ is $\delta_0$-unstable}\rcu\right|>K_2\rcu$$ 
the event that there exists $j\in \{1,\ldots, M\}$ such that $y_j$ is contained in more than $K_2$ pairs which are ${\delta_0}$-unstable, where $K_2>0$ is a constant to be specified. Then we can choose $c_2,K_2>0$ with $\P(A^\p)\leq c_2s^{-3}$.
Indeed, for any pairwise distinct $j_1,\ldots,j_{K_2}\in \{2,\ldots,M\}$ the probability that all pairs $\{y_1,y_{j_1}\},\ldots, \{y_1,y_{j_{K_2}}\}$ are $\delta_0$-unstable is at most 
$$\(\(\lambda \nu_2\(A_{\delta_0}\)+2\delta_0\(\log s\)^2\)/\(\lambda \(\log s\)^2\)\)^{K_2},$$
where $\nu_2$ denotes the Lebesgue measure in $\R^2$.
The latter probability is in $O\(s^{-4}\)$, provided that $K_2>0$ is chosen sufficiently large, while the number of possible choices of pairwise distinct $\{j_1,\ldots,j_{K_2}\}$ is in $O\(M^{K_2}\)$. Hence, $\P(A^\p)\leq c_2s^{-3}$, as desired.

Furthermore, we claim the existence of $c_3,K_3>0$ (independent of $s$) such that the probability that the number of ${\delta_0}$-unstable pairs formed by the points $y_1,\ldots,y_M$ placed uniformly at random in $Q_{\log s}(o)\times[0,\lambda]$ exceeds $K_3$ is at most $c_3s^{-3}$. Indeed, suppose that for $i\geq1$ the points $y_1,\ldots,y_i$ are already placed and denote by $A_i$ the event that $y_{i+1}$ forms a ${\delta_0}$-unstable pair with one of $y_1,\ldots, y_i$. Then the conditional probability of $A_i$ given the location of the points $y_1,\ldots,y_i$ is at most $c_4^\p M{\delta_0}\leq c_4\(\log s\)^3{\delta_0}$ for suitable constants $c_4,c_4^\p>0$. Now observe that if $A^\p$ does not occur and there exist more than $K_3$ pairs which are ${\delta_0}$-unstable, then the event $A_i$ occurs for at least $K_3/K_2$ values of $i$. Thus, the probability of obtaining more than $K_3$ unstable pairs is at most
$$(c_1+c_2)s^{-3}+M^{K_3/K_2}c_4^{{K_3/K_2}}\(\log s\)^{3{K_3/K_2}}{\delta_0}^{{K_3/K_2}},$$ so that it suffices to choose any $K_3$ with ${\gamma_0}K_3/K_2>3$. 
\end{proof}
It is easy to visualize that a small perturbation of squares centered at the nodes of a ${\delta_0}$-unstable pair could also change the connectivity properties of neighboring squares. For instance, a small perturbation of the spatially $\delta$-unstable white squares in Figure~\ref{figFED} could lead to the disconnectedness of the two black squares even if these are not spatially $\delta$-unstable. A naive approach would be to say $\{y,y^\p,y^\pp\}$ forms a ${\delta_0}$-unstable triple if $\{y,y^\prime\}$ constitutes a ${\delta_0}$-unstable pair and $\min(\left|z-z^\pp\right|_\infty,\left|z^\p-z^\pp\right|_\infty)\leq2$, where we recall our convention that $z,z^\p$ and $z^\pp$ denote the spatial coordinates of $y,y^\p$ and $y^\pp$, respectively. However, as $\pi_1(X)\subset\T^2$ forms a homogeneous Poisson process with intensity $\lambda=50\lfloor\log s\rfloor $, at least heuristically we would expect that for a fixed ${\delta_0}$-unstable pair there exist $\asymp\log s $ nodes $y^\pp$ such that $\{y,y^\p,y^\pp\}$ forms a (naively) ${\delta_0}$-unstable triple. A direct adaptation of the coupling trick of~\cite{bbVoronoi} will not work if the number of unstable triples is that large. 
\begin{figure}[!htpb]
\hspace{1cm}
\begin{minipage}[t]{0.45\linewidth}
\centering
\begin{tikzpicture}[scale=0.4]
\clip (-5,-5) rectangle (5,6);
\fill[black!20!white]
(-2,-2) rectangle (2,2);
\draw 
(-2,-2) rectangle (2,2);
\fill[black!20!white]
(0,-1) rectangle (4,3);
\draw 
(0,-1) rectangle (4,3);
\fill[white]
(-0.5,-4) rectangle (3.5,0);
\draw 
(-0.5,-4) rectangle (3.5,0);
\fill[white]
(-1.5,1) rectangle (2.5,5);
\draw 
(-1.5,1) rectangle (2.5,5);
\end{tikzpicture}
\caption{Configuration sensitive  to small perturbation of white squares}
\label{figFED}
\end{minipage}
\begin{minipage}[t]{0.45\linewidth}
\centering
\begin{tikzpicture}[scale=0.4]
\fill[black!20!white]
(0,-4) rectangle (4,0);
\draw 
(0,-4) rectangle (4,0);
\fill[white]
(-3,-2) rectangle (1,2);
\draw 
(-3,-2) rectangle (1,2);
\fill[white]
(-2,-1) rectangle (2,3);
\draw 
(-2,-1) rectangle (2,3);
\fill[white]
(-1,-3) rectangle (3,1);
\draw 
(-1,-3) rectangle (3,1);
\begin{scope}
\clip (0,-3) rectangle (3,0);
\draw[dashed]
(0,-3.5) rectangle (3.5,0);
\end{scope}

\begin{scope}
\clip (-0.1,-4) rectangle (0,1,0);
\draw[dashed]
(0,-4) rectangle (4,0);
\end{scope}

\coordinate[label=-90:$P_0$] (b) at (0,-0);
\coordinate[label=45:$P$] (b) at (3,-0);
\coordinate[label=0:$z^\p$] (b) at (1.9,-2);
\coordinate[label=0:$z$] (b) at (1,-1);
\fill (0,-0) circle (3pt);
\fill (3,-0) circle (3pt);
\fill (2,-2) circle (3pt);
\fill (1,-1) circle (3pt);
\end{tikzpicture}
\caption{$y$ is boundary-visible from $y^\p$}
\label{figFe1}
\end{minipage}
\end{figure}

Therefore, we take only those neighbors into account whose connectivity properties could be destroyed by a small perturbation of the unstable pair. 
To be more precise, we formalize the notion of leaves visible at locations close to the boundary of another leaf.
Let $\varphi=\{x_n\}_{1\leq n\leq N}=\{(z_n,t_n,\sigma_n)\}_{1\leq n\leq N}$ be an arbitrary subset of $\T^2\times[0,\infty)\times\{\pm1\}$
and for $P\in\T^2$ denote by $\varphi_P$ the set of leaves whose interior covers the point $P$, i.e.,
$$\varphi_{P}=\{x_n\in \varphi:  P\in Q(z_n)\setminus\partial Q(z_n)\},$$
 where $\partial Q(z_n)$ denotes the topological boundary of $Q(z_n)$ in $\T^2$. For $x=(z,t,\sigma),\, x^\p=(z^\prime,t^\prime,\sigma^\prime)\in\varphi$ we say that \emph{$x$ is boundary-visible from $x^\p$ in $\varphi$} (or also that \emph{$y=(z,t)$ is boundary-visible from $y^\p=(z^\p,t^\p)$ in $\varphi$}) if there exists a corner $P_0$ of $\partial Q(z^\p)$ and a point $P\in\partial Q(z^\p)\cap \partial Q(z)$ such that $P_0\in Q(z)$ and $t=\mathsf{height}_{\varphi_{P_0}}(P)$. An illustration of this definition is shown in Figure~\ref{figFe1}, where only the right-most white square is boundary-visible from the black square. First, we note that for given $x_0$ the number of $x\in\varphi$ which are boundary-visible from $x_0$ is rather small.

\begin{lemma}
\label{tailBoundLem}
Let $X$ be as above and let $X^\p$ be a point process in $\T^2\times [0,\infty)\times \lcu\pm1\rcu$ that is defined on the same probability space as $X$ and satisfies $X^\p\subset X$ a.s.
\Fm, let $x_0\in \T^2\times[0,\lambda]\times\{\pm1\}$ be arbitrary. Then \tes a constant $c>0$ (not depending on $s$, $x_0$ or the distribution of $X^\p$) \st
$$\P\(\#\{x\in X^\p:x\text{ is boundary-visible from $x_0$ in $X^\p\cup\{x_0\}$}\}\ge c\log s/\log\log s\)\le s^{-3}.$$
\end{lemma}
\begin{proof}
Without loss of generality, we may assume $z_0=o$, \sot $x_0=(o,t_0)$. \Fm, it suffices to prove a corresponding bound for the number $N^\p$ of $x\in X^\p$ \st $P_0=(-1/2,-1/2)$ and \st for $P= [P_0,P_0+(1,0)]\cap\partial Q(z)$ we have $t=\mathsf{height}_{X^\p_{P_0}}(P)$. \Ob that for $n\ge1$ if $N^\p\ge n$ then \te $x_1,x_2,\ldots, x_n\in X\cap Q_3(o)\times[0,\lambda]$ satisfying 
$$t_n\le t_{n-1}\le \cdots\le t_1\quad\text{ and }\quad z_n^{(1)}\le z_{n-1}^{(1)}\le \cdots\le z_1^{(1)},$$
where $z^{(1)}_i$ denotes the first coordinate of $z_i$. Therefore, using the Slivnyak-Mecke formula \wc
\begin{align*}
\P( N^\p\ge n) &\le3^n\(\int_{-3/2}^{3/2}\int_{-3/2}^{z^{(1)}_1}\cdots\int_{-3/2}^{z^{(1)}_{n-1}} 1dz_n^{(1)}\cdots dz_2^{(1)} dz_1^{(1)}\)\(\int_0^\lambda\int_0^{t_{1}}\cdots\int_{0}^{t_{n-1}} 1dt_n\cdots dt_2 dt_1\) \\
&= 3^n\(\int_0^3\int_{0}^{\xi_1}\cdots\int_{0}^{\xi_{n-1}} 1d\xi_n\cdots d\xi_2d\xi_1\)\frac{\lambda^n}{n!} \\
&={9^{n}\lambda^n}/{n!^2}.
\end{align*}
Using Stirling's formula we conclude that \tes $c_1>0$ \st the latter expression is most $c_1^n\lambda^n/n^{2n}$ for all sufficiently large $n$. \Ip, choosing $n=c\log s/\log\log s$, we obtain \fa  sufficiently large $s$ that
\begin{align*}
\P\(N_1\ge n\)&\le c_1^n\({50\(\log\log s\)^2}/({c^2\log s})\)^n\\
&=\exp\(n\(-\log\log s+\log\(50c_1/c^2\)+2\log\log\log s\)\)\\
&\le\exp\(-({n\log\log s})/{2}\).
\end{align*}
The last expression equals $\exp\(-({c\log s})/{2}\)$, \sot choosing $c=6$ proves the claim.
\end{proof}
For $\delta_0>0$ and $y\in\pi_{1,2}(\varphi)$ define $\varphi^y_{{\delta_0}}$ to be the set of all $y^\p$ in $\pi_{1,2}(\varphi)$ such that $\{y,y^\p\}$ does not form a ${\delta_0}$-unstable pair. For $x,x^\p,x^\pp\in\varphi$ we say that $\{x,x^\p,x^\pp\}$ forms a \emph{${\delta_0}$-unstable triple with respect to $\varphi$} (or also that $\{y,y^\p,y^\pp\}$ forms a \emph{${\delta_0}$-unstable triple with respect to $\varphi$}) if $\{y,y^\p\}$ forms a $\delta_0$-unstable pair and $y^\pp$ is boundary-visible from $y$ in $\varphi^{y}_{\delta_0}$. If the locally finite set $\varphi$ is understood, we also just say \emph{${\delta_0}$-unstable triple} instead of ${\delta_0}$-unstable triple with respect to $\varphi$.

By a \emph{${\delta_0}$-bad component} of $\varphi=\{x_n\}_{1\leq n\leq N}$ we denote a connected component of the graph with vertex set $\varphi$ and where an edge is drawn between $x_1,x_2\in\varphi$ if $\{x_1,x_2\}$ forms a ${\delta_0}$-unstable pair or if there exists a ${\delta_0}$-unstable triple containing $x_1$ and $x_2$. As in \cite[Lemma 6.4]{bbVoronoi} we first prove that ${\delta_0}$-bad components are rather small (to be more precise, of sub-logarithmic size).
\begin{lemma}
\label{etaBound}
Let $X$ be as above, let ${\eta,\gamma_0}>0$ be arbitrary and let ${\delta_0}={\delta_0}(s)\leq  s^{-{\gamma_0}}$. For $\eta>0$ denote by $C^{(2)}_{s,{\gamma_0},\eta}$ the event that no ${\delta_0}$-bad component of $X$ consists of more than $\eta\log s$ vertices. Then for all $\eta,\gamma_0>0$ the event $C^{(2)}_{s,{\gamma_0},\eta}$ occurs whp.
\end{lemma}
\begin{proof}
By Palm calculus it suffices to prove the existence of a constant $c_1>0$ such that for any fixed $(z_0,t_0,\sigma_0)\in\T^2\times[0,\lambda]\times\{\pm1\}$ the probability that the ${\delta_0}$-bad component of $X\cup \{(z_0,t_0,\sigma_0)\}$ containing $(z_0,t_0,\sigma_0)$ consists of more than $\eta\log s$ vertices is in $O(s^{-5/2})$. Next, observe that to prove this statement it suffices to consider the confetti process on a $\log s$-square centered at $z_0$. Indeed, \fa $\eta>0$ sufficiently small, if the ${\delta_0}$-bad component containing $x$ consists of at most $\eta\log s$ vertices then its $d_\infty$-diameter is at most $\log s$. 

By Lemma~\ref{etaLem1} \te  constants $c_1,K>0$ \st \fa sufficiently large $s>0$ the number of $\delta_0$-unstable pairs in $\pi_{1,2}(X)\cap \(Q_{\log s}(z_0)\times[0,\lambda]\)$ is bounded above by $K$ with probability at least $ 1- c_1s^{-3}$. Furthermore, by Lemma~\ref{tailBoundLem} and Palm calculus \te $c_2,c_3>0$ \st for all $s$ sufficiently large the probability that for all $y\in \pi_{1,2}(X)\cap \(Q_{\log s}(z_0)\times[0,\lambda]\)$ \te at most $c_2\log s/\log\log s$ elements $y^\p\in \pi_{1,2}(X)$ such that $y^\p$ is boundary-visible from $y$ in $X^y_{\delta_0}$ is at least $1-c_3s^{-5/2}$. Combining the two observations completes the proof.
\end{proof}
\begin{remark}
\label{generalApproach}
The above argument could be used to show that the $\delta_0$-bad components are of order $O(\log s/\log\log s)$. This approach was motivated by a suggestion of an anonymous referee. Originally she/he noted that although the number of leaves falling into the cube $Q(o)\times[0,\lambda]$ is $\asymp\log s$, the number of leaves that are visible in the window $Q(o)$ (and not only at its boundary) should be much smaller. For instance if one could show that the probability that more than $k$ leaves are visible decays as $k^{-ck}$ then we could deduce that whp only $O(\log s/\log\log s)$ are relevant. This tail behavior is rather easily established in dimension $1$ due to the linear order of leaves. However, it seems not completely trivial to prove the analog in dimension $2$, since configurations of visible square-shaped leaves do not form a linear but a tree-like structure. In particular, if the tree structure is very balanced and binary, naive estimates do not seem strong enough to yield tail estimates strictly stronger than exponential. As hinted by the referee if this estimate could be made rigorous, it would lead to substantial simplifications of the proof.
\end{remark}

Next, we describe the construction of the coupling in Lemma~\ref{pubCoupLem}. Recall that $\delta=\delta(s)=(4\lceil s^{\gamma}/4\rceil)^{-1}$ and $\delta_1=\lceil\delta^{-1/2}\rceil^{-1}$.
Define $K_1=100s^2\lambda\delta_1^{-3}$ and subdivide $\T^2\times[0,\lambda]$ into $K_1$ cubes $Q_i$ of the form $Q_i=Q_{\delta_1}(z)\times [(k-1)\delta_1,k\delta_1]$ for some $z\in\delta_1\Z^2$ and $k\in\{1,\ldots,\lambda/\delta_1\}$. Denote by $N$ a Poisson random variable with mean $100s^2\lambda$ and independently for each $j\in\{1,\ldots,N\}$ choose an index $s_j$ uniformly at random from $\{1,\ldots,K_1\}$. Then for $i\in\{1,2\}$ the unmarked point process $Y_i=\{y_{i,1},\ldots,y_{i,N}\}\subset\T^2\times[0,\lambda]$ corresponding to $X_i$ is constructed by choosing $y_{i,j}$ uniformly at random from $Q_{s_j}$. Our task is to provide a suitable coupling of $(y_{1,j},\sigma_{1,j})$ and $(y_{2,j},\sigma_{2,j})$ for all $j\in\{1,\ldots,N\}$.

Write $\delta_2=\sqrt{\delta_1}$.
For $a,b\in\{1,\ldots,N\}$ we say that $\{a,b\}$ forms a \emph{potentially $\delta_2$-unstable pair} if it is possible to find $\wt{y}_a\in Q_{s_a}$ and $\wt{y}_b\in Q_{s_b}$ such that $\{\wt{y}_a,\wt{y}_b\}$ forms a $\delta_2$-unstable pair. Here we do not necessarily assume $\wt{y_a},\wt{y_b}\in Y_1$ or $\wt{y_a},\wt{y_b}\in Y_2$. Similarly for $a,b,c\in\{1,\ldots,N\}$ we say that $\{a,b,c\}$ forms a \emph{potentially $\delta_2$-unstable triple} if it is possible to find $\wt{y}_j\in Q_{s_j}$, $j\in\{1,\ldots,N\}$ such that $\{\wt{y}_a,\wt{y}_b,\wt{y}_c\}$ forms a $\delta_2$-unstable triple in $\{\wt{y}_j\}_{1\leq j\leq N}$. Finally, we say that $C\subset \{1,\ldots,N\}$ forms a \emph{potentially $\delta_2$-bad component} if $C$ is a connected component in the graph with vertex set $\{1,\ldots,N\}$ and where for $a,b\in\{1,\ldots,N\}$, $a$ and $b$ are connected by an edge if $\{a,b\}$ forms a potentially $\delta_2$-unstable pair or if there exists a potentially $\delta_2$-unstable triple containing both $a$ and $b$.  
\begin{lemma}
\label{potBadLem}
For $\eta>0$ denote by $C^{(3)}_{s,\eta}$ the event of the absence of potentially $\delta_2$-bad components consisting of more than $\eta\log s$ vertices. Then for every $\eta>0$ the event $C^{(3)}_{s,\eta}$ occurs whp.
\end{lemma}
\begin{proof}
First, observe that if $\{y_a,y_b\}$ forms a $\delta_2$-unstable pair, then for all $\wt{y_a}\in Q_{s_a}$, $\wt{y_b}\in Q_{s_b}$ the pair $\{\wt{y_a},\wt{y_b}\}$ forms a $(\delta_2+2\delta_1)$-unstable pair. \Fm, \su that $\{y_a,y_b,y_c\}$ forms a $\delta_2$-unstable triple.
\Ip, assume that $y_c$ is boundary-visible from $y_a$ in $(Y_1)^{y_a}_{\delta_2}$ and for $j\in\{1,\ldots,N\}$ let $\wt{y_j}\in Q_{s_j}$ be arbitrary. Then we claim that $\wt{y_a}$ and $\wt{y_c}$ are contained in the same $(\delta_2+2\delta_1)$-bad component with respect to $\{\wt{y_j}\}_{1\leq j\leq N}$. Using Lemma~\ref{etaBound}, this will show that whp no potentially $\delta_2$-bad component contains more than $\eta\log s$ vertices.

Without loss of generality we may assume that $z_c$ lies to the north-west of $z_a$ and that a point $P\in\partial Q(z_a)\cap\partial Q(z_c)$ satisfying the assumption in the definition of the boundary-visibility property lies on the upper horizontal boundary of $Q(z_a)$. Furthermore, write $\wt{P}=P+\pi_1(\wt{z_c}-z_c)e_1+\pi_2(\wt{z_a}-z_a)e_2$, where $e_1=(1,0)$ and $e_2=(0,1)$. First, we assert $\left|\wt{P}-\wt{z_c}\right|_\infty\leq1/2$. Indeed, as $\left|\pi_1(\wt{P}-\wt{z_c})\right|=1/2$,
\begin{align*}
\left|\wt{P}-\wt{z_c}\right|_\infty> 1/2&\iff\left|\pi_2(\wt{P}-\wt{z_c})\right|> 1/2\\
&\iff \left|\pi_2(P-z_c)+\pi_2(z_c-\wt{z_c})+\pi_2(\wt{z_a}-z_a)\right|>1/2.
\end{align*}
\Ip, $\left|\wt{P}-\wt{z_c}\right|_\infty>1/2$ implies $1/2-2\delta_1\leq \left|\pi_2(P-z_c)\right|\leq 1/2$, so that $\{z_a,z_c\}$ forms a spatially $2\delta_1$-unstable pair, thereby contradicting $y_c\in (Y_1)^{y_a}_{\delta_2}$. This completes the proof of the assertion and we may henceforth use that $\left|\wt{P}-\wt{z_c}\right|_\infty=1/2$. Similarly, $\left|\wt{P}-\wt{z_a}\right|_\infty=1/2$.

Denote by $\wt{P_0}$ the upper-left corner of $Q(\wt{z_a})$. Under these assumptions, note that if $\{\wt{y_a},\wt{y_c}\}$ does not form a $(\delta_2+2\delta_1)$-unstable pair, then the set of all $d\in\{1,\ldots,N\}$, $d\ne a$ with 
\begin{enumerate}
\item $\wt{y_d}\in(\{\wt{y_j}\}_{1\leq j\leq N})^{\wt{y_a}}_{\delta_2+2\delta_1}$,
\item $\left|\wt{P_0}-\wt{z_d}\right|_\infty\leq1/2$,
\item $\wt{t_d}\le\wt{t_c}$ and 
\item $\pi_1(\wt{z_d})\ge\pi_1(\wt{z_c})$
\end{enumerate}
is non-empty, as $d=c$ is an admissible choice. Among all these values choose $d$ \st $\wt{t_d}$ is minimal. See Figure~\ref{figFe2} for an illustration of the configuration. In particular, $\wt{y_d}$ is boundary-visible from $\wt{y_a}$ so that $\wt{y_a}$ and $\wt{y_d}$ are contained in the same $(\delta_2+2\delta_1)$-bad component. This completes the proof of the claim if $d=c$, so we may assume from now that $d\neq c$. 

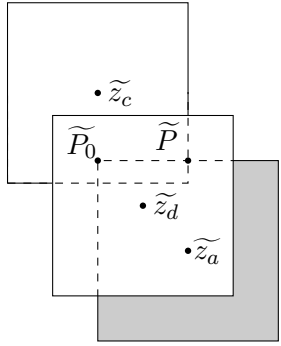
\begin{figure}[!htpb]
\centering
\begin{tikzpicture}[scale=0.6]
\fill[black!20!white]
(0,-4) rectangle (4,0);
\draw 
(0,-4) rectangle (4,0);
\fill[white]
(-2,-0.5) rectangle (2,3.5);
\draw 
(-2,-0.5) rectangle (2,3.5);
\fill[white]
(-1,-3) rectangle (3,1);
\draw 
(-1,-3) rectangle (3,1);
\begin{scope}
\clip (0,-3) rectangle (3,0);
\draw[dashed]
(0,-3.5) rectangle (3.5,0);
\end{scope}

\begin{scope}
\clip (-2,-0.6) rectangle (2,1.5);
\draw[dashed]
(-2,-0.5) rectangle (2,3.5);
\end{scope}

\coordinate[label=-90:$\wt{P_0}$] (b) at (-0.35,1.0);
\coordinate[label=135:$\wt{P}$] (b) at (2,-0);
\coordinate[label=0:$\wt{z_a}$] (b) at (1.9,-2);
\coordinate[label=0:$\wt{z_d}$] (b) at (1,-1);
\coordinate[label=0:$\wt{z_c}$] (b) at (0,1.5);
\fill (0,-0) circle (2pt);
\fill (2,-0) circle (2pt);
\fill (2,-2) circle (2pt);
\fill (1,-1) circle (2pt);
\fill (0,1.5) circle (2pt);
\end{tikzpicture}
\caption{Configuration in the proof of Lemma~\ref{potBadLem}}
\label{figFe2}
\end{figure}

If $\{\wt{y_c},\wt{y_d}\}$ forms a $2\delta_1$-unstable pair, then $\wt{y_c}$ and $\wt{y_d}$ (and thereby also $\wt{y_a}$ and $\wt{y_c}$) are contained in the same $(\delta_2+2\delta_1)$-bad component with respect to $\{\wt{y_j}\}_{1\leq j\leq N}$.  Hence, we may assume $\wt{t_d}+2\delta_1\le\wt{t_c}$ and $\pi_1(\wt{y_d})\ge \pi_1\(\wt{y_c}\)+2\delta_1$. But then $\pi_1(y_d)\ge \pi_1(y_c)$ and $t_d\le t_c$ contradicting the assumption $y_c$ is boundary-visible from $y_a$. This proves the claim that $\wt{y_a}$ and $\wt{y_c}$ are contained in the same $(\delta_2+2\delta_1)$-bad component with respect to $\{\wt{y_j}\}_{1\leq j\leq N}$.

\end{proof}
For $C$ a potentially $\delta_2$-bad component, we say that $B(C)$ occurs for $\{(y_{1,j},\sigma_{1,j})\}_{1\leq j\leq N}$ if there exist $a,b\in C$ such that $\{y_{1,a},y_{1,b}\}$ forms a $16\delta$-unstable pair. To prove Lemma~\ref{pubCoupLem} we need the following probabilistic result.
\begin{lemma}
\label{privCoupLem}
There exists a coupling between $\{(y_{1,j},\sigma_{1,j})\}_{1\leq j\leq N}$ and $\{(y_{2,j},\sigma_{2,j})\}_{1\leq j\leq N}$ such that the probability of occurrence of the following events tends to $1$ as $s\to\infty$.
\begin{enumerate}
\item $\left|y_{1,i}-y_{2,i}\right|_\infty\leq \delta_1$ for all $i\in\{1,\ldots,N\}$,
\item $\sigma_{2,i}\geq\sigma_{1,i}$ for all $i\in\{1,\ldots,N\}$,
\item if $C$ is a potentially $\delta_2$-bad component and $a,b\in C$ are \st $\sigma_{1,a}=1$ and $\sigma_{2,b}=-1$, then $B(C)$ does not occur for $\{(y_{1,j},\sigma_{1,j})\}_{1\leq j\leq N}$ and $y_{1,j}=y_{2,j}$ for all $j\in C$.
\end{enumerate}
\end{lemma}
\begin{proof}
Choose $\eta>0$ \st $2s^{-\gamma/3}\le s^{\eta\log(p_2-p_1)}$ holds \fa sufficiently large $s>0$, with $\gamma$ the constant used in the definition of $\delta=\delta(s)$ in the paragraph preceding Lemma~\ref{pubCoupLem}. Also assume that $C^{(3)}_{s,\eta}$ holds, where we recall from Lemma~\ref{potBadLem} that this happens whp. To construct the desired coupling between $\{(y_{1,j},\sigma_{1,j})\}_{1\leq j\leq N}$ and $\{(y_{2,j},\sigma_{2,j})\}_{1\leq j\leq N}$ we first define a preliminary version that may be considered as \emph{natural coupling}. This natural coupling is constructed simply by choosing $y_{1,j}=y_{2,j}$ to be uniformly distributed in $Q_{s_j}$. Furthermore, we choose $\sigma_{1,j}=1$ with probability $p_1$ and the value of $\sigma_{2,j}$ is determined conditionally on the value of $\sigma_{1,j}$. If $\sigma_{1,j}=1$ then we also put $\sigma_{2,j}=1$, but if $\sigma_{1,j}=-1$ then we put $\sigma_{2,j}=1$ with probability $(p_2-p_1)/(1-p_1)$. Starting from this simple coupling we construct the final coupling on each of the potentially $\delta_2$-bad components separately.

Let $C\subset\{1,\ldots,N\}$ be a potentially $\delta_2$-bad component. For any $\delta_0>0$ we denote by $F_{\delta_0}\subset \T^2\times\R$ the set of all $(z,t)\in\T^2\times\R$ such that $\{(z,t),(o,0)\}$ forms a $\delta_0$-unstable pair.
It is easy to check that there exists a constant $c_1>0$ such that for all $j\in\{1,\ldots, K_1\}$ and all $\wt{y}\in Q_{j}$ we have $\nu_3((\wt{y}+F_{16\delta})\cap Q_j)\leq c_1\delta_1^2\delta$, where $\nu_3$ denotes the Lebesgue measure in $\R^3$ and where we recall that $K_1=100s^2\lambda\delta_1^{-3}$. As $C$ consists of at most $\eta\log s $ vertices, we see that the expected number of $16\delta$-unstable pairs in $C$ is at most $\(\eta\log s\)^2\cdot c_1\delta_1^2\delta/\delta_1^3\leq s^{-\gamma/3}$ for all sufficiently large values of $s$.

On the other hand, denote by $G(C)$ the event that in the natural coupling we have $\sigma_{2,j}=1$ and $\sigma_{1,j}=-1$ for all $j\in C$, \sot
\begin{align*}
\P(G(C))&\geq (p_2-p_1)^{\eta\log s}= s^{\eta\log(p_2-p_1)}\geq 2s^{-\gamma/3}.
\end{align*}
In particular, we obtain $\P(G(C)\setminus B(C))\geq \P(B(C))$. The final coupling is now constructed by using the cross-over coupling described in \cite{bbVoronoi}. For the convenience of the reader we briefly recall this technique. 

Let $X_1^*,X_2^*$ be random variables with marginals distributed as $X_1,X_2$ and that are coupled according to the natural coupling described above. Choose $G^\p(C)\subset G(C)\setminus B(C)$ with $\P(G^\p(C))=\P(B(C))$ and a measure-preserving bijection $f_C$ that maps $B(C)\cup G^\p(C)$ to itself and where $B(C)$ is mapped into $G^\p(C)$ and vice versa. The existence of $f_C$ and $G^\p(C)$ is a consequence of the observation that our probability space is a non-atomic standard probability space (see, e.g.~\cite[pp. 42-43]{svind}).

Then we put $X_1=X_1^*$ and define $X_2$ by 
\begin{align*}
X_2(\omega)&=\begin{cases}X_2^*(\omega) &\text{ if }\omega\not\in B(C)\cup G^\p(C),\\
X_2^*(f_C(\omega)) &\text{ if } \omega\in B(C)\cup G^\p(C).
\end{cases}
\end{align*}
In particular, the final coupling has the following properties.
\begin{enumerate}
\item If $\omega\not\in B(C)\cup G^\prime(C)$, then $y_{1,j}=y_{2,j}$ for all $j\in C$ and there are no $16\delta$-unstable pairs in $C$.
\item If $\omega\in B(C)$ then $\sigma_{2,j}=1$ for all $j\in C$.
\item If $\omega\in G^\p(C)$ then $\sigma_{1,j}=-1$ for all $j\in C$.
\end{enumerate}
As the $\delta_2$-bad components define a partition of $\{1,\ldots,N\}$ and the cubes $Q_{s_j}$, $j\in\{1,\ldots,N\}$ form a subdivision of $\T^2\times[0,\lambda]$ the above construction yields the desired coupling.
\end{proof}
The second step in the proof of Lemma~\ref{pubCoupLem} is to show that given a coupling as described in Lemma~\ref{privCoupLem}, the occurrence of a $\psi_{X_1}$-black horizontal crossing of $R_1$ implies the existence of a $\psi_{X_2^{2\delta}}$-black horizontal crossing of $R_2$. After having constructed the explicit coupling this is a completely elementary geometric (i.e., deterministic) problem. Unfortunately, it turns turns out to be a rather delicate issue and the remainder of this section is devoted to its proof.

Lemma~\ref{privCoupLem} yields finite sets $\{(y_{i,n},\sigma_{i,n})\}_{1\leq n\leq N}\subset\T^2\times[0,\lambda]\times\{\pm1\}$, $i\in\{1,2\}$ \st
\begin{enumerate}
\item $\left|y_{1,j}-y_{2,j}\right|_\infty\leq\delta_1$ for all $j\in\{1,\ldots,N\}$,
\item $\sigma_{1,j}\leq\sigma_{2,j}$ for all $j\in\{1,\ldots,N\}$,
\item for all potentially $\delta_2$-bad components $C\subset\{1,\ldots,N\}$, if there exist $a,b\in C$ with $\sigma_{1,a}=1$ and $\sigma_{2,b}=-1$, then $B(C)$ does not occur for $\{(y_{1,j},\sigma_{1,j})\}_{1\leq j\leq N}$ and $y_{1,j}=y_{2,j}$ for all $j\in C$.
\end{enumerate}
Furthermore, without loss of generality, we may assume $C^{(1)}_s$, where we recall from Section~\ref{unifSec} that this event is defined by $\T^2\subset\bigcup_{z\in \pi_1(X)}{Q_{1/2}(z)}$. Our goal is to prove that properties (i)-(iii) above imply that if there exists a $\psi_{\varphi_1}$-black horizontal crossing of $R_1$, then there exists a $\psi_{\varphi_2^{2\delta}}$-black horizontal crossing of $R_2$, where $\varphi_1=\{(y_{1,n},\sigma_{1,n})\}_{1\leq n\leq N}$ and $\varphi_2=\{(y_{2,n},\sigma_{2,n})\}_{1\leq n\leq N}$. 

For $C\subset\{1,2,\ldots,N\}$ we denote by $D(C)$ the event that there exist $n_1,n_2\in C$ with $\sigma_{1,n_1}=1$ and $\sigma_{2,n_2}=-1$. Define a sequence $\Delta=\{\delta^{(n)}\}_{1\leq n\leq N}$ by $\delta^{(n)}=\delta_1+4\delta$ if $n$ is contained in a potentially $\delta_2$-bad component $C$ such that $D(C)$ does not occur and put $\delta^{(n)}=4\delta$ otherwise. By Lemma~\ref{privCoupLem} the occurrence of $D(C)$ implies that there are no $m_1,m_2\in C$ such that $\{y_{1,m_1},y_{1,m_2}\}$ forms a spatially $16\delta$-unstable pair. Next, we consider two results formalizing the intuition that $\delta^{(n)}$-instabilities can only occur under severe restrictions on the leaf colors.
\begin{lemma}
\label{badSizeLem}
Let $n_1,n_2\in\{1,\ldots,N\}$ and suppose that $\{z_{1,n_1},z_{1,n_2}\}$ forms a spatially $(2\delta^{(n_1)}+2\delta^{(n_2)})$-unstable pair. Denote by $C$ the potentially $\delta_2$-bad component containing $n_1$ and $n_2$. Then $D(C)$ does not occur. In particular, $\sigma_{1,n_1}=\sigma_{2,n_2}$. 
\end{lemma}
\begin{proof}
Assume the contrary, i.e., that $D(C)$ occurs. Then $\delta^{(n_1)}=\delta^{(n_2)}=4\delta$, so that $\{z_{1,n_1},z_{1,n_2}\}$ forms a spatially $16\delta$-unstable pair contradicting the observation before the lemma.
\end{proof}

\begin{lemma}
\label{scott}
Let $m,n\in\{1,\ldots,N\}$ be such that $\left|z_{1,m}-z_{1,n}\right|_\infty\leq2$ and $t_{1,m}-\delta^{(m)}<t_{1,n}+\delta^{(n)}$. Denote by $\wt{C}\subset\{1.\ldots,N\}$ the union of the potentially $\delta_2$-bad connected components containing $m$ and $n$. Furthermore, assume the existence of $m^\p,n^\p\in\wt{C}$ with $\sigma_{1,m^\p}=1$ and $\sigma_{2,n^\p}=-1$. Then $t_{1,m}<t_{1,n}$.
\end{lemma}
\begin{proof}
Assume the contrary. Then $\left|t_{1,m}-t_{1,n}\right|\leq \delta^{(m)}+\delta^{(n)}\leq \delta_2$ so that $m$ and $n$ are contained in the same potentially $\delta_2$-bad component $C\subset\{1,\ldots,N\}$. In particular, the existence of $m^\p,n^\p\in C$ with $\sigma_{1,m^\p}=1$ and $\sigma_{2,n^\p}=-1$ implies $D(C)$. But then $\delta^{(m)}=\delta^{(n)}=4\delta$, so that $\{t_{1,m},t_{1,n}\}$ forms a temporally $8\delta$-unstable pair contradicting the discussion before Lemma~\ref{badSizeLem}.
\end{proof}
To obtain a $\psi_{\varphi_2^{2\delta}}$-black horizontal crossing of $R_2$ we construct a further coloring $\psi^\Delta$ that is black-dominated by $\psi_{\varphi_2^{2\delta}}$ and show the existence of a black horizontal crossing of $R_2$ in this new coloring. To be more precise, we define 
$$\varphi^\Delta=\{\wt{x_n}\}_{1\leq n\leq N}=\{(z_{1,n},t_{1,n}+\delta^{(n)}\sigma_{2,n},\sigma_{2,n})\}_{1\leq n\leq N},$$
 where the leaf attached to $(z_{1,n},t_{1,n}+\delta^{(n)}\sigma_{2,n},\sigma_{2,n})$ is given by $Q_{1-2\sigma_{2,n}\delta^{(n)}}(o)$ (and also put $\psi^\Delta=\psi_{\varphi^\Delta}$). In other words, the shrinking/expansion and delay/advancement is more pronounced for leaves contained in a potentially $\delta_2$-bad connected component $C$ for which the event $D(C)$ does not occur.

We claim that the coloring $\psi_{\varphi_2^{2\delta}}$ black-dominates the coloring $\psi^\Delta$, so that we only need to establish the existence of a $\psi^\Delta$-black horizontal crossing of $R_2$. To prove this claim, it suffices to show that for every $n\in\{1,\ldots,N\}$,
\begin{enumerate}
\item $Q_{1-2\delta^{(n)}}(z_{1,n})\subset Q_{1-4\delta}(z_{2,n})$ and $t_{1,n}+\delta^{(n)}\ge t_{2,n}+2\delta$ if $\sigma_{2,n}=1$, and 
\item $Q_{1+2\delta^{(n)}}(z_{1,n})\supset Q_{1+4\delta}(z_{2,n})$ and $t_{1,n}-\delta^{(n)}\le t_{2,n}-2\delta$ if $\sigma_{2,n}=-1$. 
\end{enumerate}

Denote by $C$ the potentially $\delta_2$-bad connected component containing $n$. To prove the first claim, we distinguish two cases. If $D(C)$ occurs, then $\delta^{(n)}=4\delta$ and $y_{1,n}=y_{2,n}$, \sot the claim follows from the obvious relations $Q_{1-8\delta}(z_{1,n})\subset Q_{1-4\delta}(z_{1,n})$ and $t_{1,n}+4\delta\ge t_{1,n}+2\delta$. Similarly, if $D(C)$ does not occur, then $\delta^{(n)}=\delta_1+4\delta$ and $\left|y_{1,n} -y_{2,n}\right|_\infty\le \delta_1$, \sot the claim follows from $Q_{1-2\delta_1-8\delta}(z_{1,n})\subset Q_{1-4\delta}(z_{2,n})$ and $t_{1,n}+\delta_1+4\delta\ge t_{2,n}+2\delta$. The second claim can be proven by analogous arguments.

Furthermore, for $n\in\{1,\ldots,N\}$ we write $\varphi_1(n)=\{x_{1,n}\}\cup \{x_{1,m}\in\varphi_1:\sigma_{1,m}=-1,\,t_{1,m}<t_{1,n}\}$ and similarly $\varphi^\Delta(n)=\{\wt{x_n}\}\cup \{\wt{x_m}\in\varphi^\Delta:\sigma_{2,m}=-1,\,t_{1,m}-\delta^{(m)}<t_{1,n}+\delta^{(n)}\}$ (and also put $\psi^\Delta(n)=\psi_{\varphi^\Delta(n)}$). We need a result to create $\varphi^{\Delta}(n)$-black paths from $\varphi_1(n)$-black paths.
\begin{lemma}
\label{prepLemV2}
Let $n\in\{1,\ldots,N\}$ with $\sigma_{1,n}=1$ and let $P_1,P_2\in Q_{1-2\delta^{(n)}}(z_{1,n})$ be \st 
$$\mathsf{height}_{\varphi^\Delta(n)}(P_1)=\mathsf{height}_{\varphi^\Delta(n)}(P_2)=t_{1,n}+\delta^{(n)}.$$
 Assume the existence of $P_1^\p,P_2^\p\in Q(z_{1,n})$, $m_1,m_2\in\{1,\ldots,N\}$ and $i_1,i_2\in\{1,2\}$ such that $\sigma_{1,m_1}=\sigma_{1,m_2}=1$ and \st the following items hold for all $k\in\{1,2\}$.
\begin{enumerate}
\item $P_k^\p\in\partial Q(z_{1,m_k})$ and $\left|\pi_{i_k}(P_k^\p)-\pi_{i_k}(z_{1,m_k})\right|=1/2$,
\item $\left|\pi_{i_k}(P_k^\p)-\pi_{i_k}(P_k)\right|\leq \delta^{(m_k)}$.
\end{enumerate}
Assume furthermore the existence of a $\psi_{\varphi_1(n)}$-black path $\gamma^\p:[0,1]\to\T^2$ with $\gamma^\p(0)=P_1^\p$, $\gamma^\p(1)=P_2^\p$. Then there exists a $\psi^\Delta(n)$-black path ${\gamma}:[0,1]\to\T^2$ with ${\gamma}(0)=P_1$ and ${\gamma}(1)=P_2$.
\end{lemma}
\begin{proof}
Without loss of generality we may assume that $P_1$ lies to the north-west of $P_2$. Assume the assertion of the lemma was wrong. This implies that there exist two white leaves covering the north-east and the south-west corner of the square $Q(z_{1,n})$ and that prevent the existence of a $\psi^\Delta(n)$-black path connecting $P_1$ and $P_2$. To be more precise, we obtain the existence of $m_\mathsf{ne},m_\mathsf{sw}\in\{1,\ldots, N\}$ with
\begin{enumerate} 
\item$\sigma_{2,m_\mathsf{ne}}=\sigma_{2,m_\mathsf{sw}}=-1$,
\item$\max(t_{1,m_\mathsf{ne}}-\delta^{(m_\mathsf{ne})},t_{1,m_\mathsf{sw}}-\delta^{(m_\mathsf{sw})})<t_{1,n}+\delta^{(n)}$,
\item$z_{1,m_\mathsf{ne}}-(\delta^{(m_\mathsf{ne})}+1/2)(e_1+e_2)\in\(\pi_1(P_1),\pi_1(P_2)\)\times\(\pi_2(P_2),\pi_2(P_1)\)$,
\item$z_{1,m_\mathsf{sw}}+(\delta^{(m_\mathsf{sw})}+1/2)(e_1+e_2)\in\(\pi_1(P_1),\pi_1(P_2)\)\times\(\pi_2(P_2),\pi_2(P_1)\)$,
\item$\left|z_{1,m_\mathsf{ne}}-z_{1,m_\mathsf{sw}}\right|_\infty<1+\delta^{(m_\mathsf{ne})}+\delta^{(m_\mathsf{sw})}$.
\end{enumerate}
See Figure~\ref{prepLemUU2}, left, after replacing the squares $Q^{}(z_{1,m_\mathsf{ne}})$ and $Q^{}(z_{1,m_\mathsf{sw}})$ by $Q_{1+2\delta^{(m_\mathsf{ne})}}(z_{1,m_\mathsf{ne}})$ and $Q_{1+2\delta^{(m_\mathsf{sw})}}(z_{1,m_\mathsf{sw}})$, respectively. \Ob that by Lemma~\ref{scott} we have $\max(t_{1,m_\mathsf{ne}},t_{1,m_\mathsf{sw}})<t_{1,n}$. We distinguish two cases.

First, suppose $\left|z_{1,m_\mathsf{sw}}-z_{1,m_\mathsf{ne}}\right|_\infty<1$, see Figure~\ref{prepLemUU2}, left. We claim that $P_1$ and $P_1^\p$ are contained in the same connected component of $Q(z_{1,n})\setminus(Q(z_{1,m_\mathsf{ne}})\cup Q(z_{1,m_\mathsf{sw}}))$. Observe that this set consists of precisely two connected components since we assumed $\left|z_{1,m_\mathsf{sw}}-z_{1,m_\mathsf{ne}}\right|_\infty<1$. Let us assume the contrary and furthermore -- without loss of generality -- that $i_1=1$. Since $P_1^\p$ lies in the same connected component of $Q(z_{1,n})\setminus(Q(z_{1,m_\mathsf{ne}})\cup Q(z_{1,m_\mathsf{sw}}))$ as $P_2$ we have
\begin{align*}
\pi_1(P_1^\p)-(\pi_1(z_{1,m_\mathsf{ne}})-1/2)&\geq\pi_1(P_1^\p)-(\pi_1(z_{1,m_\mathsf{sw}})+1/2)\geq0,
\end{align*}
and also
\begin{align*}
\pi_1(P_1^\p)-(\pi_1(z_{1,m_\mathsf{ne}})-1/2)=\pi_1(P_1^\p)-\pi_1(P_1)+\pi_1(P_1)-(\pi_1(z_{1,m_\mathsf{ne}})-1/2)\leq \delta^{(m_1)}+0.
\end{align*}
Hence, $\{z_{1,m_1},z_{1,m_\mathsf{ne}}\}$ forms a spatially $\delta^{(m_1)}$-unstable pair, thereby contradicting Lemma~\ref{badSizeLem}. Therefore, $P_1$ and $P_1^\p$ lie in the same connected component of $Q(z_{1,n})\setminus(Q(z_{1,m_\mathsf{ne}})\cup Q(z_{1,m_\mathsf{sw}}))$. Of course, the same is true for $P_2$ and $P_2^\p$. But this yields a contradiction to the existence of a $\psi_{\varphi_1(n)}$-black path connecting $P_1^\p$ and $P_2^\p$.

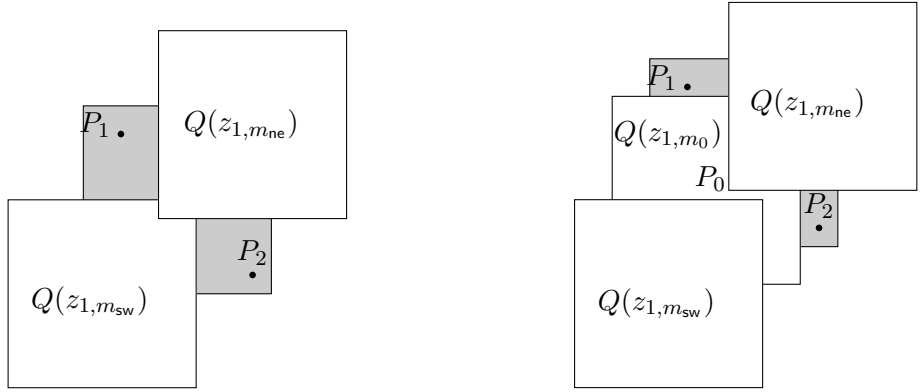
\begin{figure}[!htpb]
\hspace{1cm}
\begin{minipage}[t]{0.45\linewidth}
\centering
\begin{tikzpicture}[scale=0.25]
\fill[black!20!white]
(-5,-5) rectangle (5,5);
\draw 
(-5,-5) rectangle (5,5);
\fill[white]
(-9,-10) rectangle (1,0);
\draw (-9,-10) rectangle (1,0);
\fill[white]
(-1,-1) rectangle (9,9);
\draw (-1,-1) rectangle (9,9);
\coordinate[label=180:$P_1$] (b) at (-3,4);
\fill (-3,3.5) circle (5pt);
\coordinate[label=90:$P_2$] (b) at (4,-4);
\fill (4,-4) circle (5pt);
\coordinate[label=180:$Q(z_{1,m_\mathsf{sw}})$] (b) at (-1,-5.4);
\coordinate[label=180:$Q(z_{1,m_\mathsf{ne}})$] (b) at (7,4);
\end{tikzpicture}
\end{minipage}
\begin{minipage}[t]{0.45\linewidth}
\centering
\begin{tikzpicture}[scale=0.25]
\fill[black!20!white]
(-5,-5) rectangle (5,5);
\draw 
(-5,-5) rectangle (5,5);
\fill[white]
(-7,-7) rectangle (3,3);
\draw (-7,-7) rectangle (3,3);
\fill[white]
(-9,-12.5) rectangle (1,-2.5);
\draw (-9,-12.5) rectangle (1,-2.5);
\fill[white]
(-0.8,-2) rectangle (9.2,8);
\draw (-0.8,-2) rectangle (9.2,8);
\coordinate[label=180:$P_1$] (b) at (-3,4);
\fill (-3,3.5) circle (5pt);
\coordinate[label=90:$P_2$] (b) at (4,-4);
\coordinate[label=-135:$P_0$] (b) at (-0.4,-0.2);
\fill (4,-4) circle (5pt);
\coordinate[label=180:$Q(z_{1,m_\mathsf{sw}})$] (b) at (-1,-8);
\coordinate[label=180:$Q(z_{1,m_\mathsf{ne}})$] (b) at (7,2.6);
\coordinate[label=180:$Q(z_{1,m_0})$] (b) at (-0.6,1);
\end{tikzpicture}
\end{minipage}
\caption{Configurations in the proof of Lemma~\ref{prepLemV2}}
\label{prepLemUU2}
\end{figure}

The reasoning in the case $\left|z_{1,m_\mathsf{sw}}-z_{1,m_\mathsf{ne}}\right|_\infty\geq1$ is very similar, but we provide the details for the convenience of the reader. The first important observation is that $\{z_{1,m_\mathsf{ne}},z_{1,m_\mathsf{sw}}\}$ forms a spatially $(\delta^{(m_\mathsf{ne})}+\delta^{(m_\mathsf{sw})})$-unstable pair. Next, we denote by $P_0$ the south-west corner of the square $Q(z_{1,m_\mathsf{ne}})$. Then there exists a unique $m_0\in\{1,\ldots, N\}$,  such that $t_{1,m_0}=\mathsf{height}_{(\varphi_1)^{y_{1,m_\mathsf{ne}}}_{\delta_2}\setminus\{x_{1,m_\mathsf{ne}}\}}(P_0)$. \Ip, $y_{1,m_{0}}$ is boundary-visible from $y_{1,m_\mathsf{ne}}$ (apply the definition of boundary-visibility to the corner $P_0$ and any point $P\in\partial Q(z_{m_{1,ne}})\cap \partial Q(z_{1,m_0})$). For an illustration of the situation see Figure~\ref{prepLemUU2}, right. By Lemma~\ref{badSizeLem} we may assume $\sigma_{1,m_0}=-1$. %
Note that $Q(z_{1,n})\setminus(Q(z_{1,m_0})\cup Q(z_{1,m_\mathsf{ne}})\cup Q(z_{1,m_\mathsf{sw}}))$ consists of at most two connected components and we first assume that $P_1^\p$ and $P_2^\p$ both lie in the connected component containing the south-east corner of $Q(z_{1,n})$.
Then we have again
\begin{align*}
\pi_1(P_1^\p)-\pi_1(z_{1,m_\mathsf{ne}})+1/2\geq\pi_1(P_1^\p)-(\min\(\pi_1(z_{1,m_\mathsf{sw}}),\pi_1(z_{1,m_0})\)+1/2)-\delta^{(m_\mathsf{ne})}-\delta^{(m_\mathsf{sw})},
\end{align*}
and the latter expression is at least $-\delta^{(m_\mathsf{ne})}-\delta^{(m_\mathsf{sw})}$. \Fm, note
\begin{align*}
\pi_1(P_1^\p)-\pi_1(z_{1,m_\mathsf{ne}})+1/2=\pi_1(P_1^\p)-\pi_1(P_1)+\pi_1(P_1)-(\pi_1(z_{1,m_\mathsf{ne}})-1/2)
\leq \delta^{(m_1)}+0.
\end{align*}
From $\delta^{(m_\mathsf{ne})}=\delta^{(m_\mathsf{sw})}$ we conclude that $\{z_{1,m_\mathsf{ne}},z_{1,m_1}\}$ forms a spatially $(2\delta^{(m_1)}+2\delta^{(m_\mathsf{ne})})$-unstable pair, thereby contradicting Lemma~\ref{badSizeLem}. Since analogous arguments can be used to arrive at a contradiction in the case, where $P_1^\p$ and $P_2^\p$ both lie in the connected component containing the north-west corner of $Q(z_{1,n})$, this concludes the proof.
\end{proof}
We finally need a result allowing us to glue together paths obtained from Lemma~\ref{prepLemV2}.
\begin{lemma}
\label{prepLemV3}
Let $n_1,n_2\in\{1,\ldots,N\}$ be such that $\sigma_{1,n_1}=\sigma_{1,n_2}=1$ and $t_{1,n_1}<t_{1,n_2}$. Furthermore, let $P\in\partial Q(z_{1,n_1})$ with $\mathsf{height}_{\varphi_1}(P)=t_{1,n_1}$ and $\mathsf{height}_{\varphi_1\setminus\{x_{1,n_1}\}}(P)=t_{1,n_2}$. Then there exist $P_i\in\R^2$, $i\in\{1,2\}$ with the following properties.
\begin{enumerate}
\item  $P_i\in Q_{1-2\delta^{(n_i)}}(z_{1,n_i})$ \fa $i\in\{1,2\}$,
\item  $P_i$ is $\psi^\Delta(n_i)$-black \fa $i\in\{1,2\}$,
\item $\left|\pi_k(P_i)-\pi_k(P)\right|\leq \delta^{(n_1)}$ holds \fa $i\in\{1,2\}$, where the index $k\in\{1,2\}$ is such that $\left|\pi_k(P)-\pi_k(z_{1,n_1})\right|=1/2$,
\item the linear segment $[P_1,P_2]$ is $\psi^\Delta$-black.
\end{enumerate}
\end{lemma}
\begin{proof}
Without loss of generality we may assume that $z_{1,n_2}$ lies to the south-east of $z_{1,n_1}$ and that $P$ lies on the right vertical boundary of $Q(z_{1,n_1})$. We distinguish $5$ cases. The first two cases are devoted to configurations, where we can use arguments exploiting the $(\delta^{(n_1)}+\delta^{(n_2)})$-instability of $\{z_{1,n_1},z_{1,n_2}\}$ and the remaining cases consider configurations, where this property does not occur necessarily.
\paragraph{Case $1a$.} Denote by $P_0=z_{1,n_1}+e_1/2-e_2/2$ the south-eastern corner of the square $Q(z_{1,n_1})$. First, assume that $\{z_{1,n_1},z_{1,n_2}\}$ forms a spatially $(\delta^{(n_1)}+\delta^{(n_2)})$-unstable pair. Assume additionally that there does \emph{not} exist $n_3\in\{1,\ldots,N\}$ with $t_{1,n_3}<t_{1,n_2}$, $\left|\pi_1(P)-\pi_1(z_{1,n_3})\right|\leq1/2$, $\pi_2(z_{1,n_3})+1/2\in (\pi_2(z_{1,n_1})-1/2,\pi_2(P))$ and such that $\{z_{1,n_1},z_{1,n_3}\}$ does \emph{not} form a $(\delta^{(n_1)}+\delta^{(n_3)})$-unstable pair. For instance, the non-existence of $n_3$ is always satisfied if $\pi_2(z_{1,n_1})>\pi_2(z_{1,n_2})+1-\delta^{(n_1)}-\delta^{(n_2)}$. In particular, (due to $C^{(1)}_s$) \tes $n_4\in\{1,\ldots,N\}$ with $t_{1,n_4}= \mathsf{height}_{(\varphi_1)^{y_{1,n_1}}_{\delta_2}\setminus\{x_{1,n_1}\}}(P_0)$. Hence, $y_{1,n_4}$ is boundary-visible from $y_{1,n_1}$ and $\{z_{1,n_1},z_{1,n_2},z_{1,n_4}\}$ are contained in the same potentially $\delta_2$-bad connected component. \Ip, we conclude from Lemma~\ref{badSizeLem} that $\sigma_{2,n_4}=1$. For $j\in\{1,2\}$ define $\xi_j=(1-2\cdot1_{\pi_j(P_0)>\pi_j(z_{1,n_2})})$. Moreover, we put $P_1=(\pi_1(P_0)-\delta^{(n_1)})e_1+(\pi_2(P_0)+\delta^{(n_1)})e_2$ and $P_2=(\pi_1(P_0)+\xi_1\delta^{(n_1)})e_1+(\pi_2(P_0)+\xi_2\delta^{(n_1)})e_2$. See Figure~\ref{figpUL2a} for an illustration.
\begin{enumerate}
\item $P_i\in Q_{1-2\delta^{(n_i)}}(z_{1,n_i})$ is clear.
\item To show that $P_i$ is $\psi^\Delta(n_i)$-black assume for the sake of contradiction the existence of $m\in\{1,\ldots,N\}$ with $t_{1,m}-\delta^{(m)}<t_{1,n_i}+\delta^{(n_i)}$, $\sigma_{2,m}=-1$ and $P_i\in Q_{1+2\delta^{(m)}}(z_{1,m})$. Observe that by Lemma~\ref{scott} we have $t_{1,m}<t_{1,n_i}$. By Lemma~\ref{badSizeLem} and the non-existence of $n_3$, we see that this is impossible.
\item $\left|\pi_1(P)-\pi_1(P_i)\right|\leq \delta^{(n_1)}$ is clear.
\item  Suppose we could find $P^\p\in [P_1,P_2]$ and $m\in\{1,\ldots,N\}$ with $t_{1,m}-\delta^{(m)}<t_{1,n_4}+\delta^{(n_4)}$, $\sigma_{2,m}=-1$ and such that $P^\p\in Q_{1+2\delta^{(m)}}(z_{1,m})$. As $n_2,n_4$ are contained in the same potentially $\delta_2$-bad component we conclude from Lemma~\ref{scott} that $t_{1,m}<t_{1,n_4}$. 
But then either $\{z_{1,m},z_{1,n_1}\}$ forms a spatially $\delta_2$-unstable pair (yielding a contradiction to Lemma~\ref{badSizeLem}) or we obtain a contradiction to the assumption $t_{1,n_4}= \mathsf{height}_{(\varphi_1)^{y_{1,n_1}}_{\delta_2}\setminus\{x_{1,n_1}\}}(P_0)$.
\end{enumerate}

\begin{figure}[!htpb]
\centering
\begin{tikzpicture}[scale=0.25]
\fill[black!20!white]
(-2,-13) rectangle (8,-3);
\draw 
(-2,-13) rectangle (8,-3);
\fill[black!20!white]
(4.5,-7) rectangle (14.5,3);
\draw 
(4.5,-7) rectangle (14.5,3);
\fill[black!20!white]
(-5,-5) rectangle (5,5);
\draw 
(-5,-5) rectangle (5,5);
\fill (9.5,-2) circle (5pt);
\fill (5,-1) circle (5pt);
\fill (0,-0) circle (5pt);
\fill (5,-5) circle (5pt);
\fill (3,-8) circle (5pt);
\fill (6,-4) circle (5pt);
\fill (4,-4) circle (5pt);
\coordinate[label=0:$z_{1,n_2}$] (b) at (9.5,-2);
\coordinate[label=-90:$z_{1,n_4}$] (b) at (3,-8);
\coordinate[label=45:$P$] (b) at (5,-1);
\coordinate[label=-90:$z_{1,n_1}$] (b) at (0,-0);
\coordinate[label=-10:${P_0}$] (b) at (5,-5);
\coordinate[label=-180:$P_1$] (b) at (4,-4);
\coordinate[label=0:$P_2$] (b) at (6,-4);
\end{tikzpicture}
\caption{Configuration in case $1a$ }
\label{figpUL2a}
\end{figure}
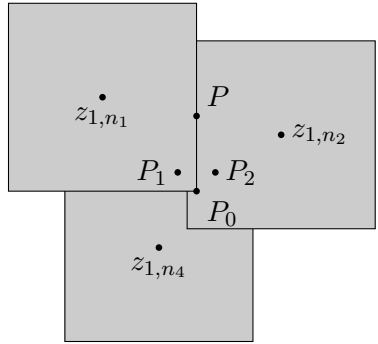

\paragraph{Case $1b$.} We assume $\pi_1(z_{1,n_2})-\pi_1(z_{1,n_1})\in (0,\delta^{(n_1)}+\delta^{(n_2)})\cup(1-\delta^{(n_1)}-\delta^{(n_2)},1)$. However, now assume additionally the existence of $n_3\in\{1,\ldots,N\}$ with $t_{1,n_3}<t_{1,n_2}$, $\left|\pi_1(P)-\pi_1(z_{1,n_3})\right|\leq1/2$, $\pi_2(z_{1,n_3})+1/2\in (\pi_2(z_{1,n_1})-1/2,\pi_2(P))$ and such that $\{z_{1,n_1},z_{1,n_3}\}$ does \emph{not} form a $(\delta^{(n_1)}+\delta^{(n_3)})$-unstable pair. Among all these values choose $n_3$ with the property that $\pi_2(z_{1,n_3})$ is maximal. Observe that $\{z_{1,n_1},z_{1,n_3}\}$ either forms a spatially $\delta_2$-unstable pair or $\{y_{1,n_1},y_{1,n_2},y_{1,n_3}\}$ forms a $\delta_2$-unstable triple. Therefore, Lemma~\ref{badSizeLem} implies $\sigma_{2,n_3}=1$. Define $\xi_1=(1-2\cdot1_{\pi_1(P_0)>\pi_1(z_{1,n_2})})$. Moreover, we put $P_1=(\pi_1(P)-\delta^{(n_1)})e_1+(\pi_2(z_{1,n_3})+1/2-\delta^{(n_3)})e_2$ and $P_2=(\pi_1(P)+\xi_1\delta^{(n_1)})e_1+(\pi_2(z_{1,n_3})+1/2-\delta^{(n_3)})e_2$. See Figure~\ref{figpUL2b} for an illustration.

\begin{enumerate}
\item $P_i\in Q_{1-2\delta^{(n_i)}}(z_{1,n_i})$ is clear.
\item To show that $P_i$ is $\psi^\Delta(n_i)$-black assume for the sake of contradiction the existence of $m\in\{1,\ldots,N\}$ with $t_{1,m}-\delta^{(m)}<t_{1,n_i}+\delta^{(n_i)}$, $\sigma_{2,m}=-1$ and $P_i\in Q_{1+2\delta^{(m)}}(z_{1,m})$. Observe that by Lemma~\ref{scott} we have $t_{1,m}<t_{1,n_2}$. By Lemma~\ref{badSizeLem} and the choice of $n_3$, we see that this is only possible if $\pi_2(z_{1,m})\in(\pi_2(z_{1,n_3})-\delta^{(m)}-\delta^{(n_3)},\pi_2(z_{1,n_3}))$. But then again $n_1,n_2,n_3$ and $m$ are contained in the same potentially $\delta_2$-bad connected component $C$, thereby yielding a contradiction to Lemma~\ref{badSizeLem}. 
\item $\left|\pi_1(P)-\pi_1(P_i)\right|\leq\delta^{(n_1)}$ is clear.
\item Suppose we could find $P^\p\in [P_1,P_2]$ and $m\in\{1,\ldots,N\}$ with $t_{1,m}-\delta^{(m)}<t_{1,n_3}+\delta^{(n_3)}$, $\sigma_{2,m}=-1$ and $P^\p\in Q_{1+2\delta^{(m)}}(z_{1,m})$. As $n_2,n_3$ are contained in the same potentially $\delta_2$-bad component we conclude from Lemma~\ref{scott} that $t_{1,m}<t_{1,n_3}$. Again by Lemma~\ref{badSizeLem} and the choice of $n_3$, we see that this is only possible if $\pi_2(z_{1,m})\in (\pi_2(z_{1,n_3})-(\delta^{(n_3)}+\delta^{(m)}),\pi_2(z_{1,n_3}))$. But again this would yield a contradiction to Lemma~\ref{badSizeLem}, since then $n_1,n_2,n_3,m$ would be contained in the same potentially $\delta_2$-bad component. 
\end{enumerate}

\begin{figure}[!htbp]
\centering
\begin{tikzpicture}[scale=0.25]
\fill[black!20!white]
(4.5,-7) rectangle (14.5,3);
\draw 
(4.5,-7) rectangle (14.5,3);
\fill[black!20!white]
(-2,-13) rectangle (8,-3);
\draw 
(-2,-13) rectangle (8,-3);
\fill[black!20!white]
(-5,-5) rectangle (5,5);
\draw 
(-5,-5) rectangle (5,5);
\fill (9.5,-2) circle (5pt);
\fill (5,-1) circle (5pt);
\fill (0,-0) circle (5pt);
\fill (3,-8) circle (5pt);
\fill (6,-4) circle (5pt);
\fill (4,-4) circle (5pt);
\coordinate[label=0:$z_{1,n_2}$] (b) at (9.5,-2);
\coordinate[label=-90:$z_{1,n_3}$] (b) at (3,-8);
\coordinate[label=45:$P$] (b) at (5,-1);
\coordinate[label=-90:$z_{1,n_1}$] (b) at (0,-0);
\coordinate[label=-180:$P_1$] (b) at (4,-4);
\coordinate[label=0:$P_2$] (b) at (6,-4);
\end{tikzpicture}
\caption{Configuration in case $1b$ }
\label{figpUL2b}
\end{figure}
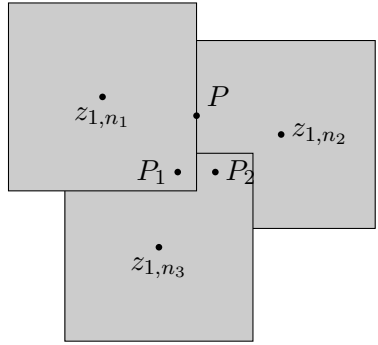
\paragraph{Case $2a$.} Henceforth, we may assume $\pi_1(z_{1,n_2})-\pi_1(z_{1,n_1})\in (\delta^{(n_1)}+\delta^{(n_2)},1-\delta^{(n_1)}-\delta^{(n_2)})$ and -- due to case $1a$ -- that $\pi_2(z_{1,n_1})-1/2<\pi_2(z_{1,n_2})+1/2-\delta^{(n_1)}-\delta^{(n_2)}$. Furthermore, assume that there does \emph{not} exist $n_3\in\{1,\ldots,N\}$ with $\sigma_{2,n_3}=-1$, $t_{1,n_3}<t_{1,n_2}$, $\left|\pi_1(P)-\pi_1(z_{1,n_3})\right|\leq1/2$, $\pi_2(z_{1,n_3})+1/2\in (\pi_2(z_{1,n_1})-1/2,\pi_2(P))$ and such that $\{z_{1,n_1},z_{1,n_3}\}$ does \emph{not} form a $(\delta^{(n_1)}+\delta^{(n_3)})$-unstable pair. Observe that in contrast to cases $1a$ and $1b$, we require also $\sigma_{2,n_3}=-1$. Then we define $P_1=(\pi_1(P)-\delta^{(n_1)})e_1+(\pi_2(z_{1,n_1})-1/2+\delta^{(n_1)})e_2$ and $P_2=(\pi_1(P)+\delta^{(n_1)})e_1+(\pi_2(z_{1,n_1})-1/2+\delta^{(n_1)})e_2$. See Figure~\ref{figpUL2c} for an illustration.
\begin{enumerate}
\item $P_1\in Q_{1-2\delta^{(n_1)}}(z_{1,n_1})$ is clear and  $P_2\in Q_{1-2\delta^{(n_2)}}(z_{1,n_2})$ follows from our assumption $\pi_2(z_{1,n_1})-1/2+\delta^{(n_1)}<\pi_2(z_{1,n_2})+1/2-\delta^{(n_2)}$. For the proof of assertion (iv) below it is also useful to note that $P_1\in Q_{1-2\delta^{(n_2)}}(z_{1,n_2})$.
\item To show that $P_i$ is $\psi^\Delta(n_i)$-black assume for the sake of contradiction the existence of $m\in\{1,\ldots,N\}$ with $t_{1,m}-\delta^{(m)}<t_{1,n_i}+\delta^{(n_i)}$, $\sigma_{2,m}=-1$ and $P_i\in Q_{1+2\delta^{(m)}}(z_{1,m})$. Observe that by Lemma~\ref{scott} we have $t_{1,m}<t_{1,n_2}$. By Lemma~\ref{badSizeLem} and the non-existence of $n_3$, we see that this is impossible.
\item $\left|\pi_1(P)-\pi_1(P_i)\right|\leq\delta^{(n_1)}$ is clear.
\item Suppose we could find $P^\p\in [P_1,P_2]$ and $m\in\{1,\ldots,N\}$ with $t_{1,m}-\delta^{(m)}<\max(t_{1,n_1}+\delta^{(n_1)},t_{1,n_2}+\delta^{(n_2)})$, $\sigma_{2,m}=-1$ and $P^\p\in Q_{1+2\delta^{(m)}}(z_{1,m})$. We conclude from Lemma~\ref{scott} that $t_{1,m}<t_{1,n_2}$. Again by Lemma~\ref{badSizeLem} and the non-existence of $n_3$, we see that this is impossible.
\end{enumerate}

\begin{figure}[!htpb]
\centering
\begin{tikzpicture}[scale=0.25]
\fill[black!20!white]
(2,-7) rectangle (12,3);
\draw 
(2,-7) rectangle (12,3);
\fill[black!20!white]
(-5,-5) rectangle (5,5);
\draw 
(-5,-5) rectangle (5,5);
\fill (7,-2) circle (5pt);
\fill (5,-1) circle (5pt);
\fill (0,-0) circle (5pt);
\fill (6,-4) circle (5pt);
\fill (4,-4) circle (5pt);
\coordinate[label=0:$z_{1,n_2}$] (b) at (7,-2);
\coordinate[label=45:$P$] (b) at (5,-1);
\coordinate[label=-90:$z_{1,n_1}$] (b) at (0,-0);
\coordinate[label=-180:$P_1$] (b) at (4,-4);
\coordinate[label=0:$P_2$] (b) at (6,-4);
\end{tikzpicture}
\caption{Configuration in case $2a$ }
\label{figpUL2c}
\end{figure}
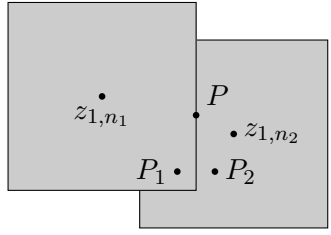
\newpage
\paragraph{Case $2b$.} Henceforth, we may assume $\pi_1(z_{1,n_2})-\pi_1(z_{1,n_1})\in (\delta^{(n_1)}+\delta^{(n_2)},1-\delta^{(n_1)}-\delta^{(n_2)})$ and -- due to case $1a$ -- that $\pi_2(z_{1,n_1})-1/2<\pi_2(z_{1,n_2})+1/2-\delta^{(n_1)}-\delta^{(n_2)}$.
 Furthermore, we may assume the existence of $n_3\in\{1,\ldots,N\}$ with $\sigma_{2,n_3}=-1$, $t_{1,n_3}<t_{1,n_2}$, $\left|\pi_1(P)-\pi_1(z_{1,n_3})\right|\leq1/2$, $\pi_2(z_{1,n_3})+1/2\in (\pi_2(z_{1,n_1})-1/2,\pi_2(P))$ and such that $\{z_{1,n_1},z_{1,n_3}\}$ does \emph{not} form a $(\delta^{(n_1)}+\delta^{(n_3)})$-unstable pair. Among all those values choose $n_3$ such that $\pi_2(z_{1,n_3})+\delta^{(n_3)}$ is maximal. Furthermore, in case $2b$ we also assume that there does \emph{not} exist $n_4\in\{1,\ldots,N\}$ with $\sigma_{2,n_4}=-1$, $t_{1,n_4}<t_{1,n_2}$, $\left|\pi_1(P)-\pi_1(z_{1,n_4})\right|\leq1/2$, $\pi_2(z_{1,n_4})-1/2\in (\pi_2(z_{1,n_3})+1/2,\pi_2(z_{1,n_3})+1/2+2\delta^{(n_3)}+2\delta^{(n_4)})$ and such that $\{z_{1,n_1},z_{1,n_4}\}$ does \emph{not} form a $(\delta^{(n_1)}+\delta^{(n_4)})$-unstable pair. Then we define $P_1=(\pi_1(P)-\delta^{(n_1)})e_1+(\pi_2(z_{1,n_3})+1/2+2\delta^{(n_3)})e_2$ and $P_2=(\pi_1(P)+\delta^{(n_1)})e_1+(\pi_2(z_{1,n_3})+1/2+2\delta^{(n_3)})e_2$. See Figure~\ref{figpUL2d} for an illustration.
\begin{enumerate}
\item $P_i\in Q_{1-2\delta^{(n_i)}}(z_{1,n_i})$ is clear, as otherwise $\{z_{1,n_2},z_{1,n_3}\}$ would form a spatially $2\delta^{(n_2)}+2\delta^{(n_3)}$ unstable pair, contradicting Lemma~\ref{badSizeLem}.
\item To show that $P_i$ is $\psi^\Delta(n_i)$-black assume for the sake of contradiction the existence of $m\in\{1,\ldots,N\}$ with $t_{1,m}-\delta^{(m)}<t_{1,n_i}+\delta^{(n_i)}$, $\sigma_{2,m}=-1$ and $P_i\in Q_{1+2\delta^{(m)}}(z_{1,m})$. Observe that by Lemma~\ref{scott} we have $t_{1,m}<t_{1,n_2}$. By Lemma~\ref{badSizeLem}, the choice of $n_3$ and the non-existence of $n_4$, we see that this is impossible.
\item $\left|\pi_1(P)-\pi_1(P_i)\right|\leq\delta^{(n_1)}$ is clear.
\item Suppose we could find $P^\p\in [P_1,P_2]$ and $m\in\{1,\ldots,N\}$ with $t_{1,m}-\delta^{(m)}<\max(t_{1,n_1}+\delta^{(n_1)},t_{1,n_2}+\delta^{(n_2)})$, $\sigma_{2,m}=-1$ and $P^\p\in Q_{1+2\delta^{(m)}}(z_{1,m})$. We conclude from Lemma~\ref{scott} that $t_{1,m}<t_{1,n_2}$. Again by Lemma~\ref{badSizeLem}, the choice of $n_3$ and the non-existence of $n_4$, we see that this is impossible.
\end{enumerate}

\begin{figure}[!htbp]
\centering
\begin{tikzpicture}[scale=0.205]
\fill[black!20!white]
(2,-7) rectangle (12,3);
\draw 
(2,-7) rectangle (12,3);
\fill[white]
(-2,-13.5) rectangle (8,-3.5);
\draw 
(-2,-13.5) rectangle (8,-3.5);
\fill[black!20!white]
(-5,-5) rectangle (5,5);
\draw 
(-5,-5) rectangle (5,5);
\fill (7,-2) circle (5pt);
\fill (5,-1) circle (5pt);
\fill (0,-0) circle (5pt);
\fill (3,-8) circle (5pt);
\fill (5.5,-2.5) circle (5pt);
\fill (4.5,-2.5) circle (5pt);
\coordinate[label=0:$z_{1,n_2}$] (b) at (7,-2);
\coordinate[label=-90:$z_{1,n_3}$] (b) at (3,-8);
\coordinate[label=45:$P$] (b) at (5,-1);
\coordinate[label=-90:$z_{1,n_1}$] (b) at (0,-0);
\coordinate[label=-100:$P_1$] (b) at (4.5,-2.5);
\coordinate[label=-80:$P_2$] (b) at (5.5,-2.5);
\end{tikzpicture}
\caption{Configuration in case $2b$  }
\label{figpUL2d}
\end{figure}
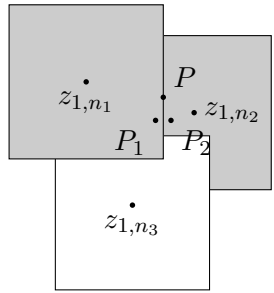

\paragraph{Case $2c$. } We can now tackle the remaining case. By the previous cases we may work under the following assumptions
\begin{itemize}
\item $\pi_1(z_{1,n_2})-\pi_1(z_{1,n_1})\in (\delta^{(n_1)}+\delta^{(n_2)},1-\delta^{(n_1)}-\delta^{(n_2)})$.
\item There exists $n_3\in\{1,\ldots,N\}$ with $t_{1,n_3}<t_{1,n_2}$, $\sigma_{2,n_3}=-1$, $\left|\pi_1(P)-\pi_1(z_{1,n_3})\right|\leq1/2$, $\pi_2(z_{1,n_3})+1/2\in (\pi_2(z_{1,n_1})-1/2,\pi_2(P))$ and such that $\{z_{1,n_1},z_{1,n_3}\}$ does \emph{not} form a $(\delta^{(n_1)}+\delta^{(n_3)})$-unstable pair. 
\item There exists $n_4\in\{1,\ldots,N\}$ with $t_{1,n_4}<t_{1,n_2}$, $\sigma_{2,n_4}=-1$, $\left|\pi_1(P)-\pi_1(z_{1,n_4})\right|\leq1/2$, $\pi_2(z_{1,n_4})-1/2\in (\pi_2(P),\pi_2(z_{1,n_2})+1/2)$.
\item Furthermore, if we choose $n_3$ such that $\pi_2(z_{1,n_3})+\delta^{(n_3)}$ is maximal and $n_4$ such that $\pi_2(z_{1,n_4})-\delta^{(n_4)}$ is minimal, then $\pi_2(z_{1,n_4})-\pi_2(z_{1,n_3})\leq 1+2\delta^{(n_3)}+2\delta^{(n_4)}$.
\end{itemize}
See Figure~\ref{figpUL2e} for an illustration. Observe that $\{z_{1,n_3},z_{1,n_4}\}$ forms a spatially $(2\delta^{(n_3)}+2\delta^{(n_4)})$-unstable pair. Furthermore, we also make the following definitions
\begin{itemize}
\item Choose $n_3^\p\in\{1,\ldots,N\}$ with $t_{1,n_3^\p}<t_{1,n_2}$, $\left|\pi_1(P)-\pi_1(z_{1,n_3^\p})\right|\leq1/2$, $\pi_2(z_{1,n_3^\p})+1/2\in [\pi_2(z_{1,n_3})+1/2,\pi_2(P))$ and $\pi_2(z_{1,n_3^\p})$ is maximal.
\item Choose $n_4^\p\in\{1,\ldots,N\}$ with $t_{1,n_4^\p}<t_{1,n_2}$, $\left|\pi_1(P)-\pi_1(z_{1,n_4^\p})\right|\leq1/2$, $\pi_2(z_{1,n_4^\p})-1/2\in (\pi_2(P),\pi_2(z_{1,n_4})-1/2]$ and $\pi_2(z_{1,n_4^\p})$ is minimal.
\end{itemize}
As $n_3$ and $n_4$ have the properties required in the definition of $n_3^\p$ respectively $n_4^\p$, this definition is indeed reasonable, i.e., we are not choosing $n_3^\p$ or $n_4^\p$ from an empty set of possible values. 
Finally, in this situation $y_{1,n_3}$, $y_{1,n_3^\p}$, $y_{1,n_4}$, $y_{1,n_4^\p}$ and $y_{1,n_1}$ are contained in the same $\delta_2$-bad connected component, contradicting Lemma~\ref{badSizeLem}.
\begin{figure}[!htbp]
\centering
\begin{tikzpicture}[scale=0.25]
\fill[black!20!white]
(2,-7) rectangle (12,3);
\draw 
(2,-7) rectangle (12,3);
\fill[white]
(-2,-11.5) rectangle (8,-1.5);
\draw 
(-2,-11.5) rectangle (8,-1.5);
\fill[white]
(-2,-0.5) rectangle (8,9.5);
\draw 
(-2,-0.5) rectangle (8,9.5);
\fill[black!20!white]
(-5,-5) rectangle (5,5);
\draw 
(-5,-5) rectangle (5,5);
\fill (7,-2) circle (5pt);
\fill (5,-1) circle (5pt);
\fill (0,-0) circle (5pt);
\fill (3,-6.5) circle (5pt);
\fill (3,4.5) circle (5pt);
\coordinate[label=0:$z_{1,n_2}$] (b) at (7,-2);
\coordinate[label=-90:$z_{1,n_3}$] (b) at (3,-6.5);
\coordinate[label=90:$z_{1,n_4}$] (b) at (3,4.5);
\coordinate[label=45:$P$] (b) at (5,-1);
\coordinate[label=-90:$z_{1,n_1}$] (b) at (0,-0);
\end{tikzpicture}
\caption{Configuration in case $2c$  }
\label{figpUL2e}
\end{figure}
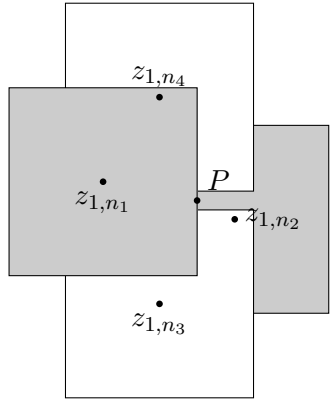

\end{proof}
After these preliminaries the existence of $\psi^\Delta$-black crossings of $R_2$ is rather immediate. Indeed, suppose that $\Gamma$ is a $\psi_{\varphi_1}$-black horizontal crossing of $R_1$. Subdivide $\Gamma$ into closed sub-paths $\Gamma_1,\ldots,\Gamma_r$ such that if we denote by $\Gamma^{o}_i$ the subset of $\Gamma_i$ obtained by deleting its two endpoints, then there exist $(z_{n_1},t_{n_1}),\dots,(z_{n_r},t_{n_r})\in \pi_{1,2}(X)$ with the property that $\mathsf{height}_{\varphi_1}(\Gamma^o_i)=\{t_{1,n_i}\}$ for all $i\in \{1,\ldots,r\}$, see also Figure~\ref{figSNB}. Let us write $\{P_i\}=\Gamma_i\cap\Gamma_{i+1}$. Applying Lemma~\ref{prepLemV3} to $P_i$ and choosing $j_0\in\{0,1\}$ with $\mathsf{height}_{\varphi_1}(P_i)=t_{1,n_{i+j_0}}$, we obtain for $j\in\{0,1\}$ points $P_{i,j}\in \mathsf{height}_{{\varphi_1^\Delta}}^{-1}(t_{1,n_{i+j}}+\delta^{(n_{i+j})})$ with the following properties:
\begin{itemize}
\item 
$\left|\pi_k(P_{i,j})-\pi_k(P_i)\right|\leq \delta^{(n_{i+j_0})}$ for all $j\in\{0,1\}$, where the index $k\in\{1,2\}$ is such that $\left|\pi_k(P_i)-\pi_k(z_{1,n_{i+j_0}})\right|=1/2$. 
\item The linear segment $[P_{i,0},P_{i,1}]$ is $\psi^\Delta$-black.
\end{itemize}
In particular, we may apply Lemma~\ref{prepLemV2} (with $n=i+1$, $P_1=P_{i,1}$, $P_2=P_{i+1,0}$, $P_1^\p=P_i$, $P_2^\p=P_{i+1}$) to obtain $\psi^\Delta$-black paths from $P_{i,1}$ to $P_{i+1,0}$. Using the $\psi^\Delta$-blackness of the joining segments $[P_{i,0},P_{i,1}]$ we can create a $\psi^\Delta$-black path starting from $P_{1,0}$ and ending at $P_{r,1}$.
\begin{figure}[!htpb]
\centering
\begin{tikzpicture}[scale=0.4]
\clip (-4,-4) rectangle (4,4);
\fill[black!20!white]
(-5,-5) rectangle (5,5);
\draw 
(-5,-5) rectangle (5,5);
\fill[black!20!white]
(1,-3) rectangle (7,3);
\draw 
(1,-3) rectangle (7,3);
\fill[black!20!white]
(-3,1) rectangle (5,7);
\draw 
(-3,1) rectangle (5,7);
\fill[white]
(-7,-5) rectangle (-1,3);
\draw 
(-7,-5) rectangle (-1,3);
\fill[white]
(-3,-7) rectangle (5,-1.7);
\draw 
(-3,-7) rectangle (5,-1.7);
\draw (0,4) .. controls (2,2) and (-2,1) .. (0.5,0);
\draw (0.5,0) .. controls (0,0) and (0,-1) .. (5,0);
\draw[decorate, decoration=brace] (1.5,4)--(1.5,1);
\draw[decorate, decoration=brace] (1,-0.5)--(-0.3,-0.5);
\draw[decorate, decoration=brace] (5,-0.5)--(1,-0.5);
\coordinate[label=0:$\Gamma_i$] (b) at (1.5,2);
\coordinate[label=0:$\Gamma_{i+1}$] (b) at (-1.1,-1.3);
\coordinate[label=0:$\Gamma_{i+2}$] (b) at (1.1,-1.3);
\fill (-0.3,1) circle (5pt);
\fill (1,-0.4) circle (5pt);
\end{tikzpicture}
\caption{Construction of the paths $\Gamma_i$ }
\label{figSNB}
\end{figure}
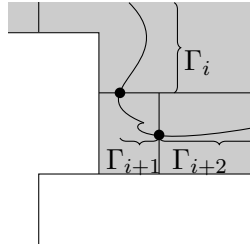
\subsection*{Acknowledgments}
The author is grateful for the detailed reports by all anonymous referees that helped to substantially improve the quality of earlier versions of the manuscript. In particular, one referee proposed the very interesting Remark~\ref{generalApproach}.
The author thanks V.~Schmidt for interesting discussions and helpful remarks. 
Furthermore, the author thanks D.~Jeulin for his lecture on the dead leaves model during the 2011 stochastic geometry summer school that marked the starting point for the current research. This work has been supported by a research grant from DFG Research Training Group 1100 at Ulm University.
\bibliography{template}
\bibliographystyle{abbrv}
\end{document}